\documentclass{amsart}
\usepackage[foot]{amsaddr}
\usepackage{amssymb}
\usepackage[swedish,english]{babel}
\usepackage{ifpdf}
\ifpdf
  \usepackage[pdftex]{graphicx}
  \usepackage{epstopdf}
\else
  \usepackage[dvips]{graphicx}
\fi
\usepackage{subfigure}
\usepackage{sidecap}
\usepackage{algorithm}
\usepackage{algpseudocode}
\usepackage[numbers,sort]{natbib}

\usepackage{graphicx}
\usepackage[T1]{fontenc}
\usepackage[utf8]{inputenc}

\usepackage[colorinlistoftodos,textsize=tiny]{todonotes}
\usepackage{diagbox}

\newcommand{\abs}[1]{\left|#1\right|}
\newcommand{\norm}[1]{\|#1\|}

\newcommand{\fdder}{\boldsymbol{\partial}}

\newcommand{\myint}[4]{\int \limits_{#1}^{#2}\! {#3} \, \mathrm{d} #4 }
\newcommand{\ip}[3]{\left( #1, #2 \right)_{#3}}

\newcommand{\dr}{\mathrm{d}}

\newcommand{\dd}[2]{\frac{\dr#1}{\dr#2}}

\newcommand{\kron}{\otimes}
\newcommand{\complex}[1]{\overline{#1}}
\newcommand{\fcoeff}[1]{\widehat{#1}}
\newcommand{\bfu}{\mathbf{u}}
\newcommand{\bfv}{\mathbf{v}}

\newcommand{\bfd}{\boldsymbol{\delta}}
\newcommand{\bfds}{\mathbf{d}}

\newcommand{\ordo}{\mathcal{O}}
\newcommand{\phat}{\widehat{P}}
\newcommand{\bferr}{\boldsymbol{\varepsilon}}

\theoremstyle{plain}
\newtheorem{theorem}{Theorem}
\newtheorem{lemma}{Lemma}

\theoremstyle{definition}

\theoremstyle{remark}
\newtheorem{remark}{Remark}

\newlength\figureheight
\newlength\figurewidth

\usepackage[pdftitle={Moving point sources},
  pdfauthor={Ljungberg Rydin, Almquist},
  pdffitwindow=true,
  breaklinks=true,
  colorlinks=true,
  urlcolor=blue,
  linkcolor=red,
  citecolor=red,
  anchorcolor=red]{hyperref}

\numberwithin{equation}{section}

\begin{document}

\title[Moving point sources in hyperbolic equations]{Approximating moving point sources in hyperbolic partial differential equations}

\author[Y. Ljungberg Rydin]{Ylva Ljungberg Rydin}
\author[M. Almquist]{Martin Almquist}

\address[Y. Ljungberg Rydin]{Department of Information Technology \\
  Uppsala University \\
  Uppsala, Sweden}

\address[M. Almquist]{Department of Information Technology \\
  Uppsala University \\
  Uppsala, Sweden}

\thanks{Corresponding author: Y.\ Ljungberg Rydin. \emph{E-mail address:}  \texttt{ylva.rydin@it.uu.se}}


\begin{abstract}
We consider point sources in hyperbolic equations discretized by finite differences. If the source is stationary, appropriate source discretization has been shown to preserve the accuracy of the finite difference method. Moving point sources, however, pose two challenges that do not appear in the stationary case. First, the discrete source must not excite modes that propagate with the source velocity. Second, the discrete source spectrum amplitude must be independent of the source position. We derive a source discretization that meets these requirements and prove design-order convergence of the numerical solution for the one-dimensional advection equation. Numerical experiments indicate design-order convergence also for the acoustic wave equation in two dimensions. The source discretization covers on the order of $\sqrt{N}$ grid points on an $N$-point grid and is applicable for source trajectories that do not touch domain boundaries.
\end{abstract}

\selectlanguage{english}

\maketitle

\section{Introduction}
Point sources are frequently used to model, for example, sources of acoustic \cite{AcousticPS} and elastic \cite{Aki2002} waves. Point sources are also used to represent boundaries and interfaces in level set methods \cite{Zahedi2010} and immersed boundary methods \cite{Peskin2002}. In many applications, the point sources are actually moving. Unless the source velocity is orders of magnitude smaller than the wave speed, accurate solution requires that the source movement is taken into account. 

Petersson et al.\ \cite{Petersson2016} derived stationary source discretizations for hyperbolic partial differential equations discretized with finite difference methods. They proved that by enforcing both moment conditions and smoothness conditions, design-order convergence is achieved away from the source. Their work builds on the work by Waldén \cite{Walden1999}, who developed theory for 1D point source discretizations for elliptic and parabolic partial differential equations. This theory was extended to source discretizations in higher dimensions by Tornberg and Engquist \cite{Tornberg2004}. A different class of point source discretizations that works well with level set methods was proposed by Zahedi and Tornberg \cite{Zahedi2010}. This type of point source has compact support in Fourier space. As a result, the source discretization is formally global in physical space. However, the source discretization coefficients decay rapidly away from the source position, which makes it possible to window the source discretization to a finite width and still satisfy a given error tolerance.

The objective of this paper is to extend the work of Petersson et al.\ \cite{Petersson2016} to moving sources. Interestingly, although their source discretization is valid for any fixed source position $x_0$, a straightforward method-of-lines discretization that uses their source fails to converge when $x_0$ is time-dependent (see Section \ref{sec:motivation}). To understand why, we repeat the accuracy analysis in Petersson et al.\ \cite{Petersson2016} with a nonzero source velocity taken into account. The analysis reveals two problems with the method-of-lines approach. First, it may excite modes whose numerical phase velocity equals the source velocity, causing a ``numerical sonic boom''. Second, the source spectrum amplitude depends on the distance between the source position and the closest grid point. As the source moves, its spectrum fluctuates with a period proportional to $h^{-1}$, where $h$ denotes the grid spacing. Our convergence proof relies on a ``motion-consistent'' source discretization with position-independent spectrum amplitude. 

This paper is organized as follows. In Section \ref{sec:motivation}, the need for a new source discretization is motivated by a comparison between different discretizations of the acoustic wave equation in one dimension. In Section \ref{sec:modelproblem}, we study the exact solution to the one-dimensional advection equation with a moving point source. Section \ref{sec:discr} introduces a spatial discretization of the advection equation. In Section \ref{sec:source}, the motion-consistent source is introduced, and design-order convergence is proved for the advection equation. Next, in Section \ref{sec:window}, we show that the motion-consistent source, which formally has global support in physical space, can be windowed to a width proportional to $\sqrt{N}$ grid points on an $N$-point grid. Section \ref{sec:implementation} describes the numerical implementation of the motion-consistent source. The convergence properties are verified by numerical experiments with the advection equation in Section \ref{sec:numexp1D}. In Section \ref{sec:numexp2D}, numerical experiments that indicate design-order convergence for the two-dimensional wave equation with an accelerating source are presented. Section \ref{sec:conclusion} concludes the work.   
 
\section{Motivation} \label{sec:motivation}
To motivate the need for a new type of point source discretization, we consider the acoustic wave equation on the real line, 
\begin{equation}
\label{eq:waveeq1d}
\begin{aligned}
\rho \dd{v}{t}+ \theta_x &= 0, \\
\frac{1}{K}\dd{\theta}{t} + v_x &= g(t) \delta(x-x_0(t)),
\end{aligned}
\end{equation}
where $\theta$ is pressure, $v$ is particle velocity, $\rho$ is density, $K$ is bulk modulus, $g(t)$ is the source time function, $\delta$ is the Dirac delta distribution, and $x_0(t)$ is the source trajectory. The wave speed is $c = \sqrt{\frac{K}{\rho}}$. In this example we set $K = 1$ and $\rho = 1$, which gives $c = 1$. We impose homogeneous initial data at $t=0$.
The smoothness of the solution is determined by the smoothness of $g$ and $x_0$. We choose the Gaussian source time function
\begin{equation} \label{eq:Gauss}
 	g(t) = \frac{1}{\sigma \sqrt{2 \pi}} e^{-\frac{(t-t_0)^2}{2\sigma^2}},
\end{equation} 
with $\sigma = \frac{1}{5}$ and $t_0 = 2$. We let the source move with constant velocity $v_0=0.3$ and set
\begin{equation}
	x_0(t) = 1.6 + v_0t.
\end{equation}
Figure \ref{fig:true} shows the pressure component of the exact solution at time $t = 2$. 

\begin{figure}[hbt]
\includegraphics[width=0.7\linewidth]{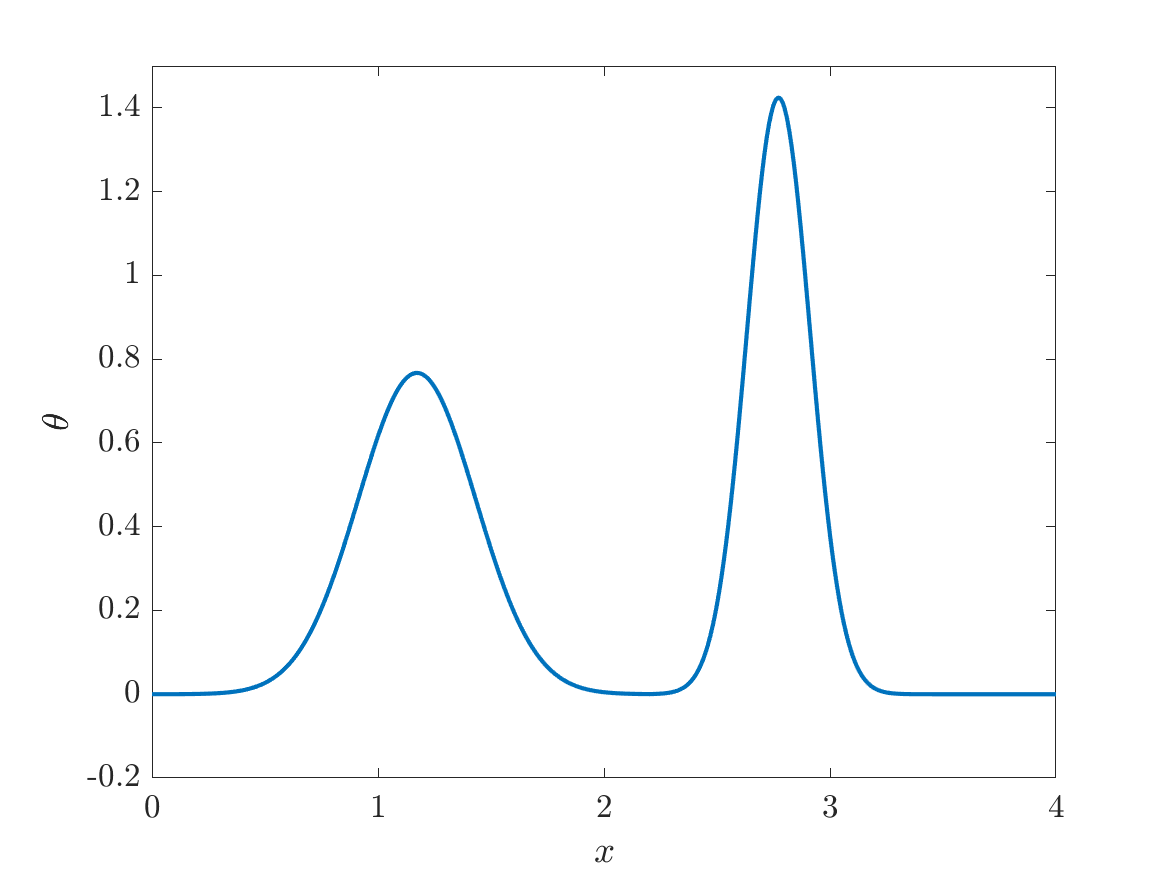}
\caption{ Exact solution of the acoustic wave equation with a moving point source and a Gaussian source time function.}
\label{fig:true}
\end{figure}
To discretize \eqref{eq:waveeq1d} we truncate the real line to the interval $[0, 4]$ and impose periodic boundary conditions. We use 400 grid points, which yields the grid spacing $h=0.01$. In what follows, we will pay particular attention to the highest mode $k$ that can be represented on the grid: the mode $\abs{kh}=\pi$, which we refer to as the $\pi$ mode. Our discretizations take the form
\begin{equation}
\label{eq:wave1dDiscrete}
\begin{aligned}
\dd{\bfv}{t} +  \fdder_+ \boldsymbol{\theta} &= 0 \\
\dd{\boldsymbol{\theta}}{t} +  \fdder_- \bfv &= g(t)  \bfd^{x_0(t)}.
\end{aligned}
\end{equation}
where $ \fdder_{\pm}$ are finite difference operators and $\bfd^{x_0(t)}$ denotes the discrete approximation of $\delta(x-x_0(t))$. If a choice of $\bfd^{x_0}$ produces a convergent method, we say that $\bfd^{x_0}$ is a \emph{motion-consistent} discretization of the $\delta$ distribution.

 We will investigate the performance of the following two finite difference (FD) methods:
\begin{itemize}
	\item[(FD 1)] \textbf{Centered stencil}. This method is given by $\fdder_+ = \fdder_- = \fdder$, where $\fdder$ denotes the standard centered fourth order finite difference operator. This method propagates the $\pi$ mode with phase velocity $0$. This implies that there is some wavenumber $k_*$, $\abs{k_*h} < \pi$, which propagates with the source velocity, as illustrated in Figure \ref{fig:disprel}. This method is therefore vulnerable to the \emph{numerical sonic boom}, which is discussed in detail in Section \ref{sec:source}.

	\item[(FD 2)] \textbf{Dual-pair stencil}. This method, introduced in \cite{Dovgilovich15}, is given by $\fdder_{\pm} = \fdder \pm B$, where $\fdder$ is a centered operator and $B=B^\top$ is chosen so that $\fdder_{\pm}$ are upwind/downwind operators. We use the fifth order upwind/downwind operators whose footprints are offset by one compared to the sixth order centered operator. The numerical phase velocity of this method is strictly larger than $v_0=0.3$ (see Figure \ref{fig:disprel}). Hence, this method avoids the numerical sonic boom in the test case studied in this section.
\end{itemize}
\begin{figure}[h]
\centering
\includegraphics[width=0.7\linewidth]{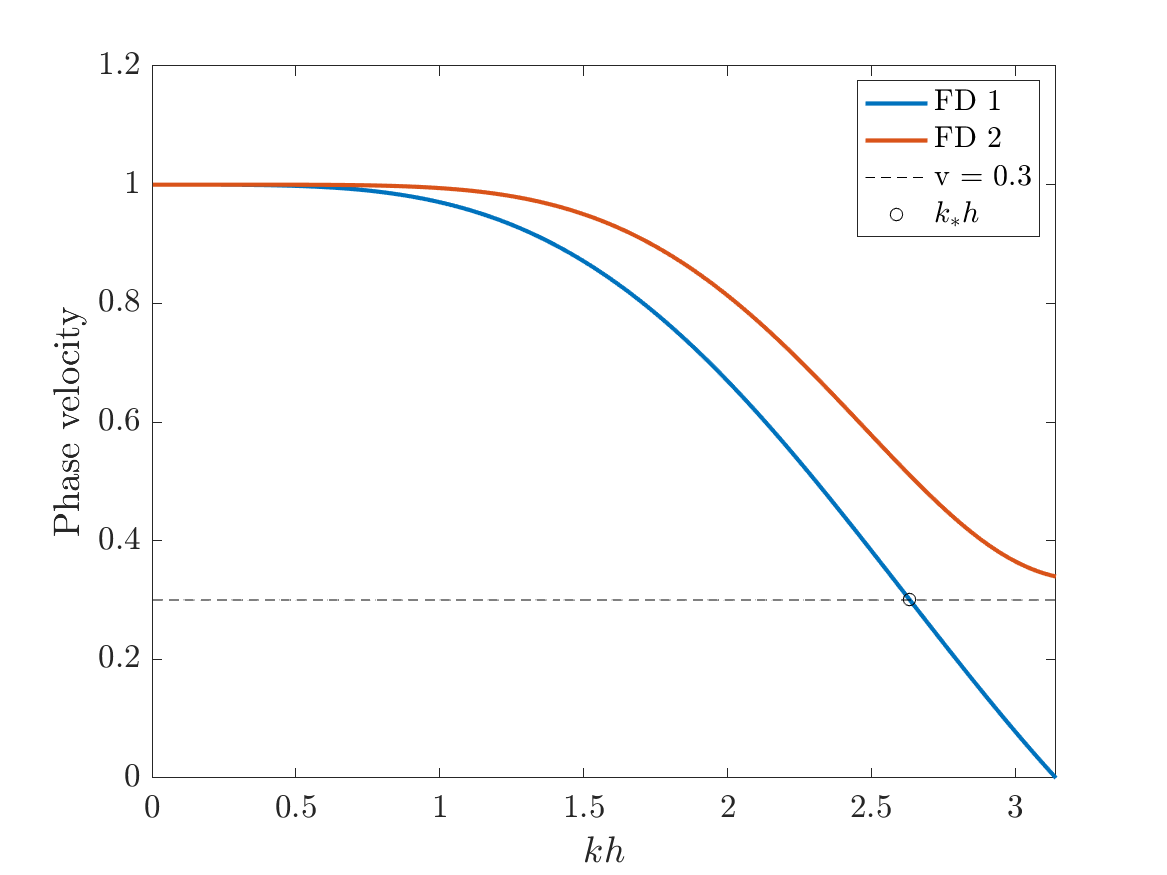}
\caption{Phase velocity for the 4$th$ order centered stencil (FD 1) and the 5$th$ order dual-pair stencil (FD 2). The sonic boom wavenumber corresponding to source velocity $v_0 = 0.3$ is indicated by a circle.}
\label{fig:disprel}
\end{figure}

The stationary source discretizations in \cite{Petersson2010, Petersson2016} are based on discrete $\delta$ distributions $\bfd^{x_0} \approx \delta(x-x_0)$ of the form
\begin{equation} \label{eq:source_phi}
	\bfd_j^{x_0} = \frac{1}{h} \phi \left(\frac{x_j-x_0}{h} \right),
\end{equation}
where the function $\phi$ has compact support. These approximations are accurate for any fixed $x_0$, but are not designed to be motion-consistent. We will refer to them as \emph{motion-inconsistent}. For $p$th order convergence with a stationary source, $\bfd^{x_0}$ must satisfy $p$ moment conditions. When combined with centered finite differences, $\bfd^{x_0}$ additionally needs to satisfy $p$ smoothness conditions, which serve to remove the $\pi$ mode from the spectrum of $\bfd^{x_0}$. When combined with the dual-pair stencil, or any other stencil that propagates the $\pi$ mode with non-zero velocity, no smoothness conditions are required (see e.g.\ \cite{OREILLY2017572} for an example with staggered operators). In this section, we investigate the following three source discretizations:

\begin{itemize}
	\item[(Source 1)] \textbf{Motion-inconsistent source with continuous $\phi$}. This source discretization from \cite{Petersson2016} satisfies 4 moment conditions and 4 smoothness conditions. The function $\phi$ is continuous but not continuously differentiable.
	\item[(Source 2)] \textbf{Motion-inconsistent source with continuously differentiable $\phi$}. This source discretization from \cite{Petersson2010} satisfies 2 moment conditions. The function $\phi$ is continuously differentiable.
	\item[(Source 3)] \textbf{Motion-consistent source}. This is the motion-consistent source, presented in Section \ref{sec:source} of this paper, with 4 moment conditions and 4 \emph{sonic boom conditions} imposed. 
\end{itemize}
The sources are illustrated in Figure \ref{fig:compoldandnewsources}.

\begin{figure}[h]
\centering
\includegraphics[width=0.7\linewidth]{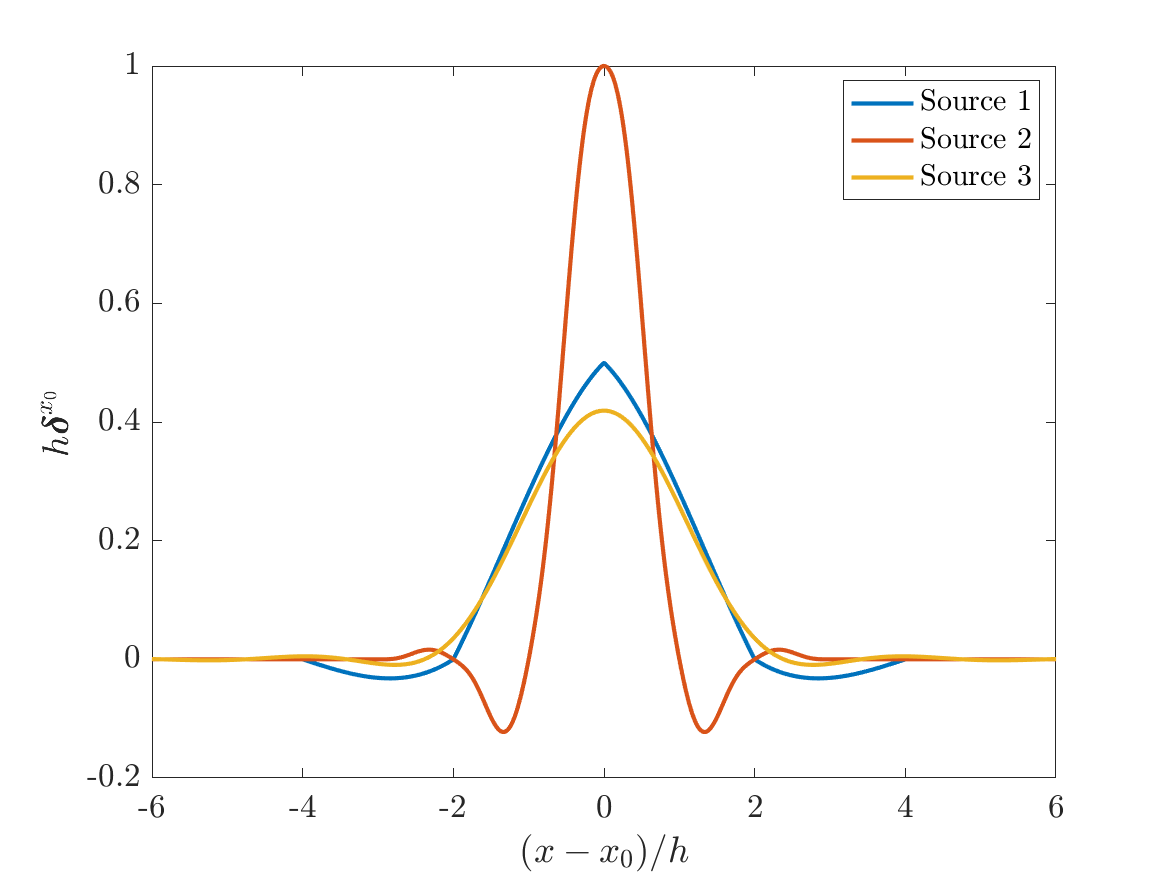}
\caption{Discrete $\delta$ distributions $\bfd^{x_0}$ for the motion-inconsistent source with continuous $\phi$ (Source 1), the motion-inconsistent source with continuously differentiable $\phi$ (Source 2), and the motion-consistent source (Source 3) }
\label{fig:compoldandnewsources}
\end{figure}

\begin{figure}[hbt]
	\subfigcapmargin = 5pt
   \subfigure[Centered finite differences, motion-inconsistent source with continuous $\phi$]{
   \label{subfig:standard_petersson_et_al}
    \includegraphics[width=0.475\textwidth]{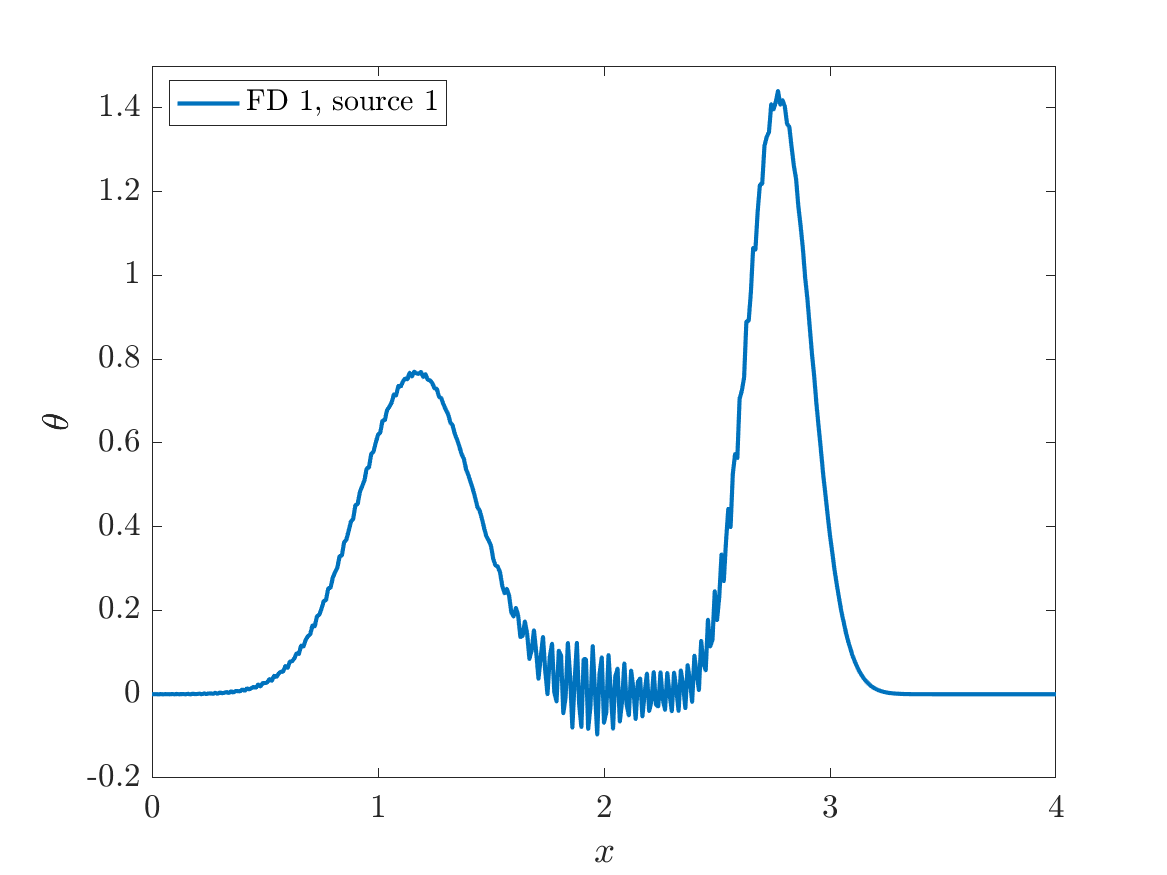}}
  \subfigure[Dual-pair finite differences, motion-inconsistent source with continuous $\phi$]{
  \label{subfig:upwind_petersson_et_al}
    \includegraphics[width=0.475\textwidth]{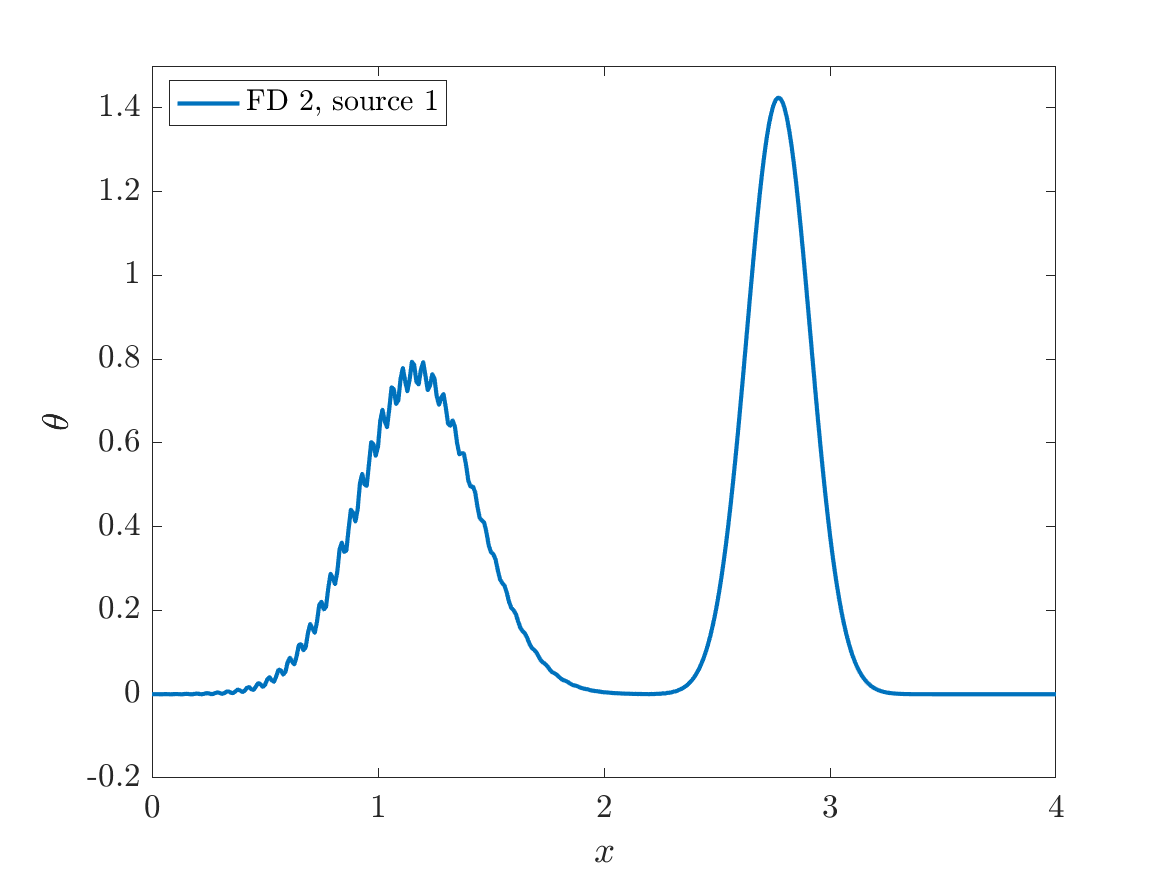}}\\
  \subfigure[Dual-pair finite differences, motion-inconsistent source with continuously differentiable $\phi$]{
  \label{subfig:upwind_petersson_sjogreen}
    \includegraphics[width=0.475\textwidth]{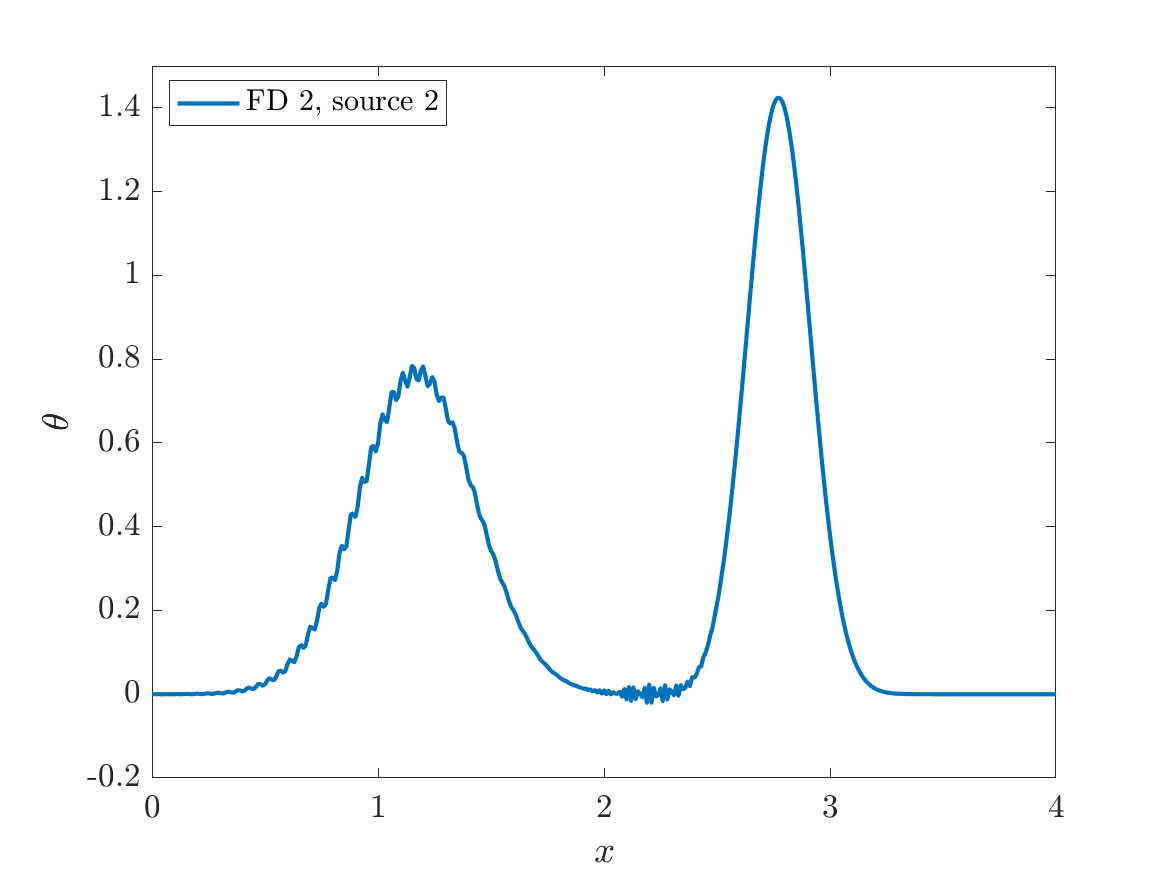}}
  \subfigure[Centered finite differences, motion-consistent source]{
  \label{subfig:FFTsource}
    \includegraphics[width=0.475\textwidth]{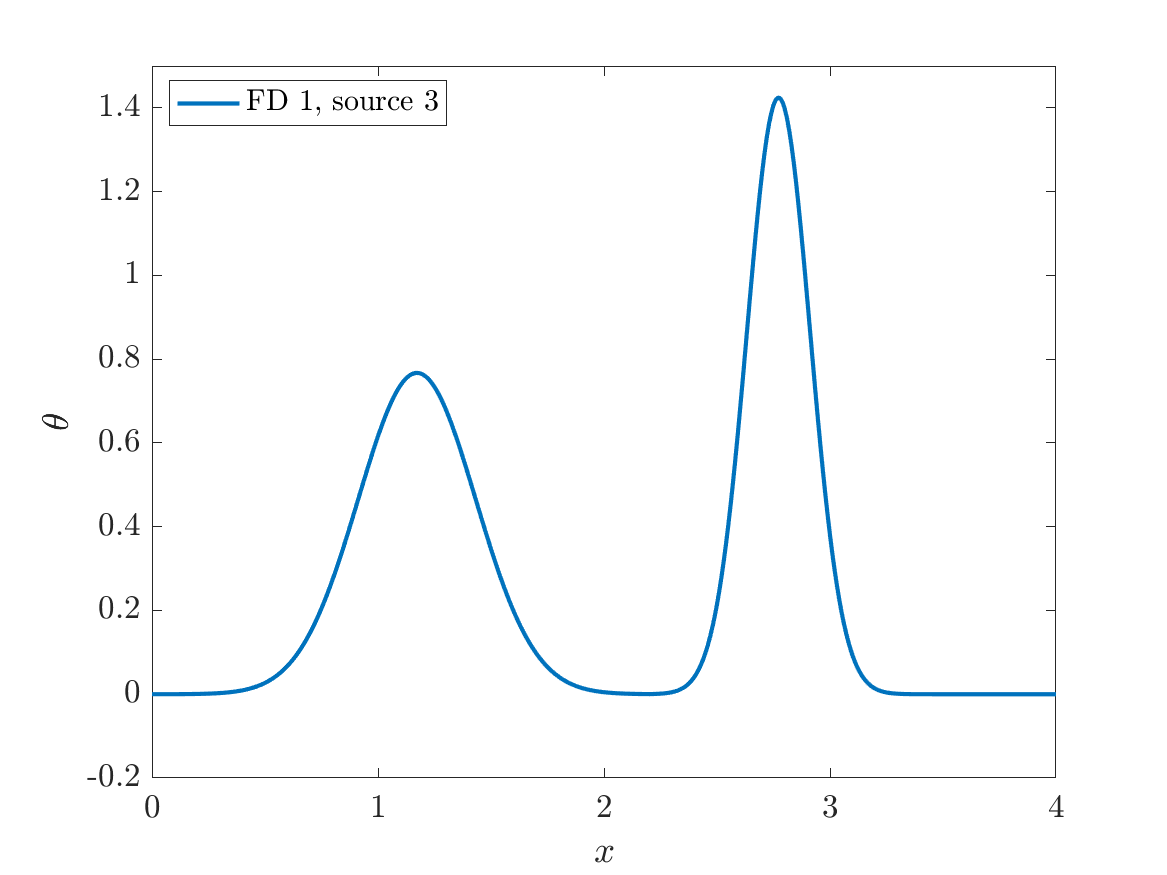}}
  \caption{Numerical solution to \ref{eq:waveeq1d} with different finite difference operators and source discretizations on the domain $x \in [0, \ 4], $ with $h = 0.01$ and $x_0 = 1.6 + 0.3 t$.}
  \label{fig:sourceComparisonWave}
\end{figure}
Figure \ref{fig:sourceComparisonWave} shows the solution obtained with the following four combinations of finite difference method and source discretization:
\begin{enumerate}
	\item \textbf{FD 1, source 1}. Figure \ref{subfig:standard_petersson_et_al} shows large artifacts in the solution. This is to some extent expected, since FD 1 is vulnerable to the numerical sonic boom.
	\item \textbf{FD 2, source 1}. Figure \ref{subfig:upwind_petersson_et_al} shows  artifacts that are smaller than in Figure \ref{subfig:standard_petersson_et_al}, but still clearly visible, even though FD 2 eliminates the numerical sonic boom issue. We conclude that there are additional requirements for motion consistency.
	\item \textbf{FD 2, source 2}. One might hypothesize that regularity of $\phi$ is important for motion consistency. However, Figure \ref{subfig:upwind_petersson_sjogreen} shows no significant improvement over Figure \ref{subfig:upwind_petersson_et_al}, despite $\phi$ being continuously differentiable. This indicates that regularity of $\phi$ is not of primary importance.
	\item \textbf{FD 1, source 3}. Figure \ref{subfig:FFTsource} shows that the motion-consistent source developed in this paper produces no visible artifacts.
\end{enumerate}
All four spatial discretizations were time-integrated with the classical fourth order Runge--Kutta method. Combination 4, with the motion-consistent source, is the only combination without clearly visible artifacts, and also the only combination that actually converges to the true solution as $h \rightarrow 0$. This motivates the need for motion-consistent source discretizations.

\section{Model problem} \label{sec:modelproblem}
To derive a provably motion-consistent source discretization, we consider the advection equation with a moving point source,
\begin{equation} \label{eq:ut_ux}
	u_t + cu_x = g(t) \delta(x - x_0(t)), \quad x \in [0, L], \quad t>0,
\end{equation}
where $u$ is the solution and $c$ the wave speed. We impose periodic boundary conditions and consider $u$ to be $L$-periodic in $x$. We assume that the source trajectory, $x_0$, and the source time funtion, $g$, are sufficiently smooth functions of $t$. We further assume that $g(t)=0$ for $t\leq 0$ and for $t\geq t_1>0$, where $t_1$ is some finite time. For simplicity we impose the homogeneous initial condition
\begin{equation}
	u(x,0) = 0, \quad x \in [0, L].
\end{equation}
While the source is active, i.e., for $t < t_1$, the exact solution may be discontinuous. We shall not attempt to prove high-order convergence for $t < t_1$. Hence, all of the subsequent convergence analysis assumes $t \geq t_1$.
We will also restrict the analysis to constant source velocity $v_0<c$, which allows us to set $x_0(t) = v_0 t$. Numerical experiments indicate that the convergence rates that we prove carry over to accelerating sources.

We will use the inner product
\begin{equation}
	\ip{u}{v}{} = \frac{1}{L} \myint{0}{L}{\complex{u}(x) v(x) }{x},
\end{equation}
where $\complex{u}$ denotes the complex conjugate of $u$. The Fourier modes, $e^{ikx}$, where
\begin{equation}
 	k \in  \left\{0, \pm \frac{2\pi}{L}, \pm \frac{4\pi}{L}, \ldots \right\} =: K_{\infty},
 \end{equation}
 are orthonormal in the inner product. The Fourier series representation of an $L$-periodic function $u$ is
 \begin{equation}
	u(x) = \sum \limits_{k \in K_{\infty}} \fcoeff{u}_k e^{i k x},
\end{equation}
where the Fourier coefficients, $\fcoeff{u}_k$, are given by
\begin{equation}
	\fcoeff{u}_k = \ip{e^{ik x}}{u}{} = \frac{1}{L} \myint{0}{L}{e^{-ikx}u(x) }{x}, \quad k \in K_{\infty}.
\end{equation}
The Fourier coefficients of the $\delta$ distribution are
\begin{equation} \label{eq:fourier_delta}
	\fcoeff{\delta}_k = \frac{1}{L} \myint{0}{L}{e^{-ikx} \delta(x-x_0) }{x} = \frac{1}{L}e^{-ikx_0}.
\end{equation}
Taking the inner product of the advection equation \eqref{eq:ut_ux} with $e^{ik x}$ leads to a system of ordinary differential equations for the Fourier coefficients of $u$,
\begin{equation}
	\dd{\fcoeff{u}_{k}}{t} + i k c \fcoeff{u}_{k} = \frac{g(t)}{L} e^{-i k x_0}, \quad k \in K_{\infty},
\end{equation}
with initial conditions $\fcoeff{u}_{k}(0)=0$.
The solution is
\begin{equation}
	\fcoeff{u}_{k}(t) = \frac{1}{L} \myint{0}{t}{g(\tau) e^{i k (c \tau - x_0(\tau) - ct)}}{\tau} .
\end{equation}
In the case of constant source velocity, $x_0(t) = v_0 t$, the solution simplifies to
\begin{equation} \label{eq:solution_cont}
	\fcoeff{u}_{k}(t) = \frac{1}{L} \myint{0}{t}{g(\tau) e^{i k ((c - v_0)\tau - c t)}}{\tau}.
\end{equation}
Consider $t\geq t_1$ so that $g(t)=0$. Assuming that $v_0 \neq c$ and that $g$ has at least $n$ continuous derivatives, applying the integration-by-parts formula $n$ times results in
\begin{equation} \label{eq:cont_ibp_n}
	\fcoeff{u}_{k}(t) = \frac{(-1)^n}{L (ik(c-v_0))^n } \myint{0}{t}{g^{(n)}(\tau) e^{i k ((c - v_0)\tau -ct)}}{\tau}.
\end{equation}
Taking the absolute value shows that the Fourier coefficients decay with $\abs{k}$ according to
\begin{equation} \label{eq:fcoeff_decay}
	\abs{\fcoeff{u}_{k}(t)} \leq \frac{C}{\abs{k}^n},
\end{equation}
for some constant $C$, which is independent of $k$. The decay of the Fourier coefficients is a consequence of the smoothness of $u$. Note however, that if $v_0=c$ so that the source moves with the wave speed, then $u$ develops the  discontinuity known as the sonic boom. In this case, the Fourier coefficients are
\begin{equation}
	\fcoeff{u}_{k}(t) = \frac{e^{-i k  c t}}{L} \myint{0}{t}{g(\tau) }{\tau}.
\end{equation}
and there is no decay with $\abs{k}$.

\section{Spatial discretization} \label{sec:discr}
The interval $[0, L]$ is discretized by an equidistant grid with $M$ points, $x_j = (j-1)h$, $j = 1,  2, \dots, M$, where the grid spacing is $h = L/M$. For simplicity, we assume that $M$ is odd such that $M=2N+1$, where $N$ is an integer. This is no actual restriction and our results hold also when $M$ is even. We use boldface font to denote grid vectors,
\begin{equation}
	\bfu = [\bfu_1, \bfu_2, \ldots, \bfu_{M}]^\top.
\end{equation}
The discrete inner product is defined as
\begin{equation}
	\ip{\bfu}{\bfv}{h} = \frac{1}{L} \sum \limits_{j=1}^{M} h \complex{\bfu}_j \bfv_j.
\end{equation}
The first $2N+1$ Fourier modes, i.e., $e^{ikx}$, where
\begin{equation}
 	k \in  \left\{0, \pm \frac{2\pi}{L}, \ldots, \pm N\frac{2\pi}{L}  \right\} =: K_N,
 \end{equation}
 are orthonormal in the discrete inner product. Thus, any grid vector can be represented by its Fourier expansion,
 \begin{equation}
	\bfu_j = \sum \limits_{k \in K_N}^{} \fcoeff{\bfu}_k e^{i k x_j},
\end{equation}
where the Fourier coefficients, $\fcoeff{\bfu}_k$, are given by
\begin{equation}
	\fcoeff{\bfu}_k = \ip{e^{ik x}}{\bfu}{h} = \frac{1}{L} \sum \limits_j h  e^{-ikx_j} \bfu_j, \quad k \in K_N.
\end{equation}

Let $\bfd^{x_0}$ be a discrete approximation of $\delta(x-x_0)$ and let $\fdder$ denote a centered finite difference operator. The semidiscrete approximation of the advection equation \eqref{eq:ut_ux} reads
\begin{equation}
	\label{eq:semidisc_ut_ux}
	\dd{\bfu}{t} + c \fdder \bfu = g(t) \bfd^{x_0} .
\end{equation}
Inserting the Fourier expansions yields
\begin{equation}
	\dd{\fcoeff{\bfu}_k}{t} + c\fcoeff{\fdder}_k \fcoeff{\bfu}_k = g(t) \fcoeff{\bfd}^{x_0}_k,
\end{equation}
where $\fcoeff{\fdder}_k$ denotes the symbol of the finite difference operator. For a $p$th order accurate centered finite difference operator, the symbol can be written as (see \cite{GustafssonKreissOliger})
\begin{equation}
	\fcoeff{\fdder}_k = ik\phat(kh),
\end{equation}
where $\phat$ is real-valued and even, and
\begin{equation} \label{eq:phat_ordo}
	\phat(kh) = 1 + \ordo((kh)^p).
\end{equation}
For notational convenience we will henceforth let the dependence on $kh$ be implied and write simply $\phat$ instead of $\phat(kh)$. The $\ordo$ notation in \eqref{eq:phat_ordo} means that there exist constants $C$ and $\kappa_0$ such that
\begin{equation}
	\abs{\phat-1} \leq C \abs{kh}^p, \quad \abs{kh} \leq \kappa_0.
\end{equation}
The $\ordo$ notation will be used frequently in subsequent sections.

The solution to the semidscrete problem \eqref{eq:semidisc_ut_ux} is
\begin{equation} \label{eq:fcoeff_solution_discrete}
	\fcoeff{\bfu}_k  = \myint{0}{t}{ g(\tau) \fcoeff{\bfd}^{x_0(\tau)}_k e^{cik\phat (\tau-t)} }{\tau}.
\end{equation}

\section{Motion-consistent source discretization} \label{sec:source}
We propose a discrete $\delta$ distribution whose Fourier coefficients satisfy
\begin{equation} \label{eq:fss_def}
	\fcoeff{\bfd}^{x_0}_k = \frac{e^{-ikx_0}}{L} F(kh),
\end{equation}
where $F$ is a bounded and sufficiently smooth real-valued even function. We will prove that $\bfd^{x_0}$ is motion-consistent. Notice that the spectrum of $\bfd^{x_0}$ is fixed in the sense that the magnitude of $\fcoeff{\bfd}^{x_0}_k$ is independent of $x_0$. It is straightforward to verify that $\bfd^{x_0}$ satisfies the shift theorem:
\begin{equation}
	\fcoeff{\bfd}_k^{x_0} = e^{-ikx_0} \fcoeff{\bfd}_k^{0}.
\end{equation}

The drawback of the motion-consistent source is that it formally has global support in physical space, which would make it impossible to use in a finite, non-periodic domain. However, we will later show that the elements of $\bfd^{x_0}$ decay rapidly in magnitude away from $x_0$. It is thus possible to apply a window $W$, of finite width, such that
\begin{equation}
	\norm{ W\bfd^{x_0} - \bfd^{x_0} }_{\infty} \leq \mu,
\end{equation}
for any given error tolerance $\mu$. In practice, we propose to use the windowed source $W\bfd^{x_0}$. First, however, we will analyze the periodic problem and show $p$th order convergence for the global source $\bfd^{x_0}$. Section \ref{sec:window} discusses how to choose $W$ so that $p$th order convergence is retained.
 
By \eqref{eq:fcoeff_solution_discrete}, the Fourier coefficients of the discrete solution obtained with the motion-consistent source are
\begin{equation}
	\fcoeff{\bfu}_k(t) = \frac{F(kh)}{L} \myint{0}{t}{ g(\tau) e^{-ikx_0(\tau)} e^{cik\phat (\tau-t)}  }{\tau}.
\end{equation}
For constant source velocity, we obtain
\begin{equation} \label{eq:fcoeff_discr_2}
	\fcoeff{\bfu}_k(t)  = \frac{F(kh)}{L} \myint{0}{t}{ g(\tau) e^{ik((c\phat-v_0)\tau -c \phat t)}  }{\tau}.
\end{equation}
Assuming that $c\phat \neq v_0$, integrating by parts $n$ times yields,
\begin{equation} \label{eq:fcoeff_discr_ibp}
	\fcoeff{\bfu}_k(t) = F(kh) \frac{(-1)^n}{L(ik(c\phat-v_0))^n} \myint{0}{t}{ g^{(n)}(\tau) e^{ik((c\phat-v_0)\tau -c \phat t)} }{\tau} ,
\end{equation}
which is a discrete analog of \eqref{eq:cont_ibp_n}. Recall that we used \eqref{eq:cont_ibp_n} to show that the Fourier coefficients of the exact solution decay with $\abs{k}$. Using \eqref{eq:fcoeff_discr_ibp} to show similar decay of the discrete Fourier coefficients is an essential part of our convergence proof. The reason why the source discretizations proposed by Petersson et al.\ are not motion-consistent is that they do not admit a result similar to \eqref{eq:fcoeff_discr_ibp}. We elaborate on this fact in Section \ref{sec:source_petersson}.

\subsection{Motion-incosistent sources} \label{sec:source_petersson}
Let $\bfds^{x_0}$ denote the discrete $\delta$ distribution proposed by Petersson et al.\ for stationary sources. Unlike the motion-consistent discretization, which we designed in Fourier space, $\bfds^{x_0}$ is designed to have minimal support in physical space. Petersson et al.\ state that the Fourier coefficients take the form
\begin{equation} \label{eq:fourier_delta_petersson}
	\fcoeff{\bfds}^{x_0}_k = \frac{e^{-ikx_0}}{L} f_{x_0}(kh),
\end{equation}
for some function $f_{x_0}$. Here we have taken the liberty of adapting their notation to the current setting by adding the subscript $x_0$, to highlight that $f_{x_0}$ may change with $x_0$. 
The source $\bfds^{x_0}$ does not satisfy the shift theorem, because
\begin{equation}
	\fcoeff{\bfds}^{x_0}_k - e^{-ikx_0} \fcoeff{\bfds}_k^{0} = \frac{e^{-ikx_0}}{L} \left(f_{x_0}(kh) - f_{0}(kh)\right).
\end{equation}
and $f_{x_0} = f_0$ does not hold, in general. The shift theorem would be satisfied if $f_{x_0} = f_0$ for any $x_0$ or, equivalently,
\begin{equation}
	\dd{}{x_0} f_{x_0}(kh) = 0.
\end{equation}

Let us now attempt to derive a result similar to \eqref{eq:fcoeff_discr_ibp} for $\bfds^{x_0}$. With constant source velocity, the Fourier coefficients of the corresonding discrete solution, which we here denote by $\fcoeff{\bfv}_k$, are
\begin{equation} \label{eq:fcoeff_discr_2_petersson}
	\fcoeff{\bfv}_k(t)  = \frac{1}{L} \myint{0}{t}{ f_{x_0}(kh) g(\tau) e^{ik((c\phat-v_0)\tau -c \phat t)}  }{\tau}.
\end{equation}
The next step is to integrate by parts. We then need to consider the time derivative of $f_{x_0}(kh)$. If the time derivative is zero, we can proceed to derive the desired result. By the chain rule, we have
\begin{equation}
	\dd{}{t} f_{x_0}(kh) = v_0 \dd{}{x_0} f_{x_0}(kh).
\end{equation}
Notice that the time derivative is zero if 
\begin{equation}
v_0=0 \quad \mbox{or} \quad \dd{}{x_0} f_{x_0}(kh) = 0 .
\end{equation}
That is, we need either a stationary source or a discrete $\delta$ distribution that satisfies the shift theorem. Notice that for a stationary source, the requirement on $f_{x_0}$ vanishes.

Since $f_{x_0}$ does not depend on the absolute position, only on the distance to the nearest grid point, we have $f_{x_0}=f_{x_0+h}$. If $\mathrm{d}/\mathrm{d}x_0 \, f_{x_0}$ is not identically zero, it follows that 
\begin{equation}
	\dd{}{x_0} f_{x_0}(kh) \propto h^{-1},
\end{equation}
and it appears impossible to bound $\mathrm{d}/\mathrm{d}x_0 \, f_{x_0}$ as $h \rightarrow 0$. We believe that these oscillations in $f_{x_0}$ are the cause of the numerical artifacts in Figures \ref{subfig:upwind_petersson_et_al} and \ref{subfig:upwind_petersson_sjogreen}.

\subsection{Moment conditions}
For a stationary source, requiring the discrete $\delta$ distribution to satisfy an appropriate number of \emph{moment conditions} ensures accuracy for low wavenumbers. The definition of moment conditions used by Petersson et al. states that $\bfds^{x_0}$ satisfies $m$ moment conditions if
\begin{equation} \label{eq:moment_def}
	\ip{x^\nu}{\bfds^{x_0}}{h} = x_0^\nu, \quad \nu = 0, \ldots, m-1.
\end{equation}
Petersson et al.\ showed that the moment conditions \eqref{eq:moment_def} imply that (assuming $m \geq 1$) 
\begin{equation} \label{eq:moment_fourier}
\begin{aligned}
	f_{x_0}(0) &= 1, \\
	f_{x_0}^{(\nu)}(0) &= 0, \quad \nu = 1, \ldots, m-1.
\end{aligned}
\end{equation}
Since the motion-consistent source is defined in terms of its Fourier coefficients, it is natural to also formulate moment conditions in the Fourier domain. Hence, rather than attempting to satisfy \eqref{eq:moment_def}, we say that $\bfd^{x_0}$ satisfies $m$ moment conditions if
\begin{equation} \label{eq:moment_fourier_fss}
\begin{aligned}
	F(0) &= 1, \\
	F^{(\nu)}(0) &= 0, \quad \nu = 1, \ldots, m-1.
\end{aligned}
\end{equation}
A consequence of $m$ moment conditions is that $F$ satisfies
\begin{equation}
	F(kh) = 1 + \ordo((kh)^m).
\end{equation}

\subsection{Sonic boom conditions}
The numerical sonic boom manifests for wavenumbers in the vicinity of $k_*$ such that $c\phat(k_*h) = v_0$, which means that the finite difference operator propagates the $k_*$ mode with the source velocity. Equation \eqref{eq:fcoeff_discr_ibp} indicates that we cannot derive a strong enough bound on $\fcoeff{\bfu}_k$ near $k=k_*$ unless $F\rightarrow 0$ sufficiently fast as $k\rightarrow k_*$. To ensure that $F\rightarrow 0$ sufficiently fast, we introduce \emph{sonic boom conditions} that we require $F$ to satisfy.

Let $k_*$ denote the smallest positive sonic boom wavenumber. We say that $F$ satisfies $s$ sonic boom conditions if $F$ is smooth on $(0, k_*h)$ and
\begin{equation}
\begin{aligned}
	F^{(\nu)}(k_*h) &= 0, \quad \nu = 0,\ldots,s-1, \\
	F(k h) &= 0, \quad k > k_*.
\end{aligned}
\end{equation}
If there are no sonic boom wavenumbers, i.e., no solutions to $c\phat(k_*h) = v_0$, then the sonic boom conditions do not imply any conditions on $F$.

\begin{remark}
It may not be strictly necessary to require $F(k h) = 0$ for $k > k_*$, but it is convenient. Wavenumbers $k>k_*$ are generally under-resolved and our experience is that removing them from the source spectrum tends to make the discrete solution smoother without causing any loss of accuracy.
\end{remark}

The sonic boom conditions may be viewed as a generalization of the \emph{smoothness conditions} introduced by Petersson et al.\ for stationary sources. The effect of $s$ smoothness conditions is that
\begin{equation} \label{eq:smoothness_petersson}
	f^{(\nu)} _{x_0}(\pi) = 0, \quad \nu = 0, \ldots, s-1.
\end{equation}
Note that Petersson et al.\ considered centered finite difference operators, for which $kh=\pi$ is the only solution to $\phat(kh)=0$ (see Figure \ref{fig:disprel}). Since the source velocity is zero, this solution corresponds to the sonic boom wavenumber, i.e., $k_*h = \pi$. The smoothness conditions \eqref{eq:smoothness_petersson} are thus the sonic boom conditions corresponding to $v_0=0$.

The sonic boom conditions allow us to prove the following lemma, which is essential to the convergence proof.

\begin{lemma} \label{lemma:symbol_at_boom}
If $F$ satisfies $s$ sonic boom conditions, then there is a constant $C$ such that
\begin{equation} \label{eq:lemma}
	\frac{\abs{F(kh)}}{\abs{c\phat - v_0}^s} \leq C, \quad \abs{k} \leq k_*.
\end{equation}
\end{lemma}
\begin{proof}
Note that
\begin{equation}
	\frac{\abs{F(kh)}}{\abs{c\phat - v_0}^s} = \frac{1}{c^s} \frac{\abs{F(kh)}}{\abs{\phat - \frac{v_0}{c}}^s} = \frac{1}{c^s} \frac{\abs{F(kh)}}{\abs{G(kh)}^s}
\end{equation}
where 
\begin{equation}
	G(kh) = \phat- \frac{v_0}{c}.
\end{equation}
Assume that $G$ has a zero at $k_*h$, of multiplicity $\alpha$. By Taylor's theorem, there are positive constants $\Delta \kappa$ and $C_G$ such that
\begin{equation} \label{eq:taylor_G}
	\abs{G(kh)} \geq C_G \abs{kh - k_*h}^{\alpha}, \quad k_*h - \Delta \kappa \leq kh \leq k_*h.
\end{equation}
Due to the sonic boom conditions and Taylor's theorem, there is a constant $C_F$ such that
\begin{equation} \label{eq:taylor_F}
	\abs{F(kh)} \leq C_F \abs{kh - k_*h}^s, \quad k_*h - \Delta \kappa \leq kh \leq k_*h.
\end{equation}
We divide the interval $[0, k_* h)$ into two subintervals: $I_1 = [0, k_*h-\Delta \kappa]$ and $I_2 = [k_*h-\Delta \kappa, k_*h)$. On $I_1$, $G$ is nonzero and hence $\abs{G} \geq G_{min} > 0$. We also have $\abs{F} \leq \norm{F}_{\infty} < \infty$, by assumption.
It follows that
\begin{equation}
	\frac{\abs{F(kh)}}{\abs{G(kh)}^s} \leq \frac{\norm{F}_\infty }{G_{min}^s} =: C_1.
\end{equation}
On $I_2$, using \eqref{eq:taylor_G} and \eqref{eq:taylor_F} yields
\begin{equation}
	\frac{\abs{F(kh)}}{\abs{G(kh)}^s} \leq \frac{C_F \abs{kh - k_*h}^s}{C_G^s \abs{kh - k_*h}^{\alpha s} }.
\end{equation}
We note that for a general multiplicity $\alpha$, $F$ would actually need to satisfy $\alpha s$ sonic boom conditions. All the finite difference operators considered in this paper are such that $\phat$ is a strictly decreasing function of $\abs{k}$, which allows us to assume $\alpha=1$. We obtain
\begin{equation}
	\frac{\abs{F(kh)}}{\abs{G(kh)}^s} \leq \frac{C_F \abs{kh - k_*h}^s}{C_G^s \abs{kh - k_*h}^{s} } = \frac{C_F}{C_G^s} =: C_2.
\end{equation}
Combining the two subintervals, we find that \eqref{eq:lemma} holds with
\begin{equation}
	C = \max{\left(\frac{C_1}{c^s}, \frac{C_2}{c^s}\right)}.
\end{equation}
\end{proof}

\subsection{Practical implementation of moment and sonic boom conditions}
While there are infinitely many choices of $F$ that satisfy $m$ moment conditions and $s$ sonic boom conditions, in our implementation we have opted to define $F$ as
\begin{equation} \label{eq:F_poly}
	F(\kappa) = \left\{ 
	\begin{array}{ll}
	Q_{m+s-1}(\kappa), & 0 \leq \kappa \leq k_*h \\
	0, & \kappa > k_*h
	\end{array}
	\right.,	
\end{equation}
where $Q_{m+s-1}$ denotes a polynomial of degree $m+s-1$. This polynomial is uniquely determined by the $m$ moment and $s$ sonic boom conditions. Since $F$ is assumed to be even, \eqref{eq:F_poly} defines $F(\kappa)$ for negative $\kappa$ as well.

The non-zero part of $F$, given by $Q_{m+s-1}$, is displayed for the cases $m=s=q$, $q = 2, 6, 14$ in Figure \ref{fig:F}. Recall that $q$ moment conditions imply that $q-1$ derivatives of $F$ are zero at $kh=0$ while $q$ sonic boom conditions imply that $q-1$ derivatives of $F$ are zero at $kh=k_*h$.
 \begin{figure}[h]
 	\centering
 	\includegraphics[width=0.7\linewidth]{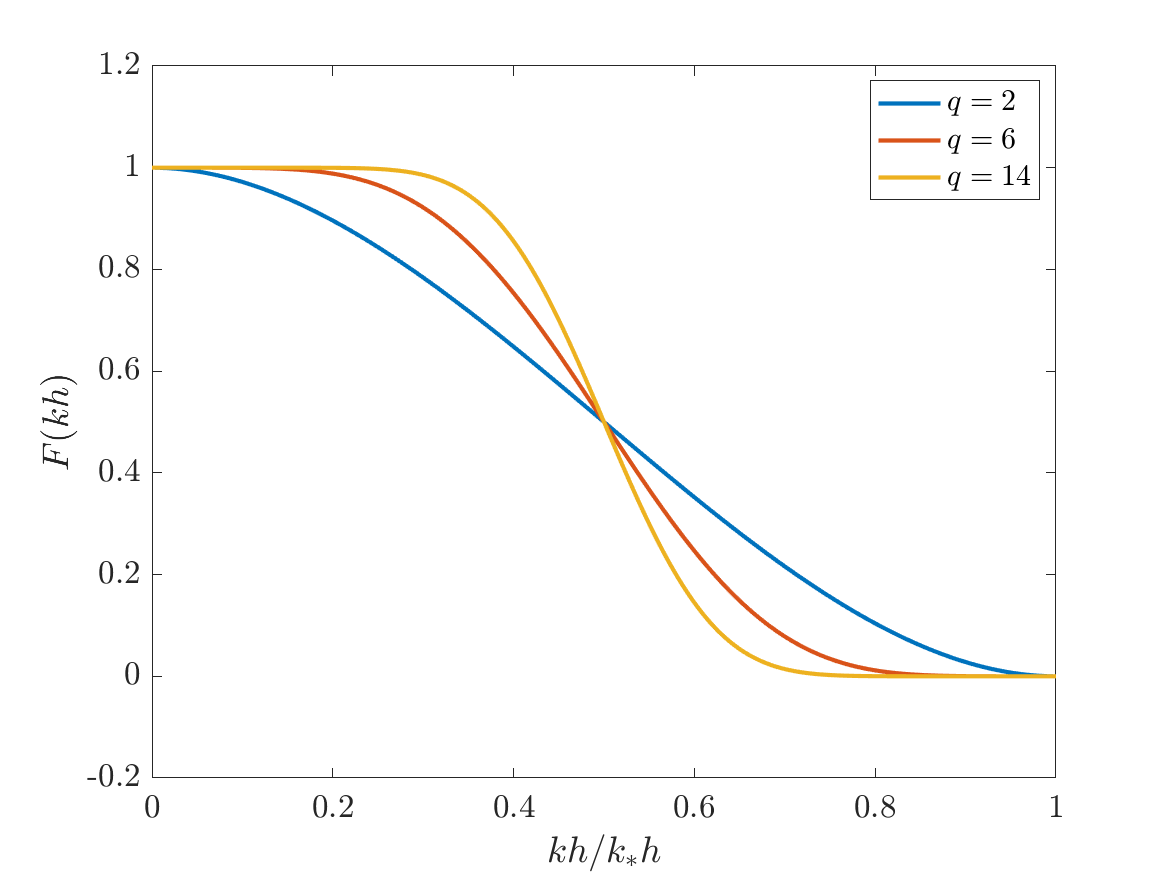}
 	\caption{Non-zero part of $F(kh)$ for the case $m=s=q$.}
 	\label{fig:F}
 \end{figure}

\subsection{Convergence result for the advection equation}
The errors in the Fourier coefficients are
\begin{equation}
	\fcoeff{\bferr}_k = \fcoeff{\bfu}_k - \fcoeff{u}_k, \quad k \in K_N.
\end{equation}
We will first bound the Fourier coefficients of the error, and then use that result to bound the error itself. The proofs are quite similiar to the proofs in Petersson et al.\ \cite{Petersson2016}, with some extensions to account for a non-zero source velocity.
\begin{theorem} \label{thm:fourier}
Let $g$ be $p+2$ times continuously differentiable and let $g(t)=0$ for $t \leq 0$ and $t \geq t_1$. If $F$ satisfies $p$ moment conditions and $p+1$ smoothness conditions, then, for any $t \geq t_1$,
\begin{equation}
	\fcoeff{\bferr}_0 = 0
\end{equation}
and
\begin{equation}
	\abs{\fcoeff{\bferr}_k} \leq C \frac{h^p}{\abs{k}}, \quad k \in K_N \setminus \{0\},
\end{equation}
for some constant $C$.
\end{theorem}

\begin{proof} We organize the proof into three cases: $k=0$, $\frac{2 \pi h}{L} \leq \abs{kh} \leq \kappa_0$, and $\abs{kh} > \kappa_0$. The constant $\kappa_0$ is not known at this point but will be specified under Case 2 below. We will frequently use $C$ to denote generic constants that are independent of $k$ and $h$.
\\ \\
\noindent \textbf{Case 1:} $k=0$. \\
Compare the exact solution \eqref{eq:solution_cont} with the discrete solution \eqref{eq:fcoeff_discr_2} and use that $F(0)=1$ and $\phat(0)=1$. It follows that $\fcoeff{\bferr}_0 = 0$.
\\ \\
\noindent \textbf{Case 2:} $\frac{2 \pi h}{L} \leq \abs{kh} \leq \kappa_0$ \\
Note that the error can be expressed as follows:
\begin{equation}
\begin{aligned}
	\abs{\bferr_k} &= \abs{\fcoeff{\bfu}_k - \fcoeff{u}_k} = \abs{\fcoeff{\bfu}_k - F(kh)\fcoeff{u}_k + (F(kh)-1)\fcoeff{u}_k} \\
	&\leq \underbrace{\abs{\fcoeff{\bfu}_k - F(kh)\fcoeff{u}_k}}_{T_1} + \underbrace{\abs{F(kh)-1} \abs{\fcoeff{u}_k} }_{T_2} \\
\end{aligned}
\end{equation}
Let us first bound $T_2$. By \eqref{eq:fcoeff_decay}, we have
\begin{equation}
	\abs{\fcoeff{u}_{k}} \leq \frac{C}{\abs{k}^n}. 
\end{equation}
By the assumption of $p$ moment conditions, we have
\begin{equation}
	\abs{F(kh)-1} = \ordo((kh)^p).
\end{equation}
It follows that
\begin{equation}
	T_2 = \abs{F(kh)-1} \abs{\fcoeff{u}_k} \leq \ordo((kh)^p) \frac{C}{\abs{k}^n}  = \frac{1}{\abs{k}^{n}}  \ordo((kh)^p).
\end{equation}
Because $g$ is $p+2$ times continuously differentiable we may set $n=p+1$, which yields 
\begin{equation}
	T_2 = \frac{1}{\abs{k}^{p+1}} \ordo((kh)^p).
\end{equation}

For $T_1$, we have
\begin{equation} \label{eq:T1_complicated}
\begin{aligned}
	T_1 &= \abs{\fcoeff{\bfu}_k - F(kh)\fcoeff{u}_k} \\
	&=\abs{\frac{F(kh)(-1)^n}{L(ik)^n} \myint{0}{t}{ g^{(n)}(\tau) \left(\frac{e^{ik((c\phat-v_0)\tau -c\phat t)} }{(c \phat - v_0)^n} - \frac{e^{ik((c-v_0)\tau - ct)}}{(c-v_0)^n} \right)}{\tau}} \\
	&\leq C \frac{\abs{F(kh)}}{\abs{k}^n} \myint{0}{t}{ \abs{g^{(n)}(\tau)} \abs{\frac{e^{ik((c\phat-v_0)\tau -c\phat t)} }{(c \phat - v_0)^n} - \frac{e^{ik((c-v_0)\tau - ct)}}{(c-v_0)^n} }}{\tau} \\
	&\leq \frac{C}{\abs{k}^n} \myint{0}{t}{ \abs{\frac{e^{ik((c\phat-v_0)\tau -c\phat t)} }{(c \phat - v_0)^n} - \frac{e^{ik((c-v_0)\tau - ct)}}{(c-v_0)^n} }}{\tau} \\
	&\leq \frac{C}{\abs{k}^n} \myint{0}{t}{ \abs{ \frac{e^{ik((c-v_0)\tau - ct)}}{(c-v_0)^n} } \abs{\frac{e^{ikc(\phat-1)(\tau-t)} (c-v_0)^n }{(c \phat - v_0)^n} - 1 }}{\tau} \\
	&\leq \frac{C}{\abs{k}^n} \myint{0}{t}{ \abs{\frac{e^{ikc(\phat-1)(\tau-t)} (c-v_0)^n }{(c \phat - v_0)^n} - 1 }}{\tau}.
\end{aligned}
\end{equation}
Next, we need to utilize that $\phat \simeq 1$. Note that
\begin{equation}
\begin{aligned}
	\frac{(c-v_0)^n }{(c \phat - v_0)^n} &= \frac{(c-v_0)^n }{(c - v_0 + c(\phat-1))^n} = \frac{1 }{\left(1 + \frac{c}{c-v_0}(\phat-1)\right)^n} = \frac{1}{(1+z)^n},
\end{aligned}
\end{equation}
where we have defined
\begin{equation}
	z := \frac{c}{c-v_0}(\phat-1) = \ordo((kh)^p).
\end{equation}
Taylor expanding yields
\begin{equation}
	\frac{1}{(1+z)^n} = 1 - nz + \frac{n(n+1)}{2} z^2 + ... = 1 + \ordo(z) = 1 + \ordo((kh)^p),
\end{equation}
where we used that $n$ does not depend on $k$ or $h$. We conclude that
\begin{equation} \label{eq:fraction_simplified}
	\frac{(c-v_0)^n }{(c \phat - v_0)^n} = 1 + \ordo((kh)^p).
\end{equation}
Substituting \eqref{eq:fraction_simplified} into \eqref{eq:T1_complicated} yields
\begin{equation} \label{eq:pre_beta}
\begin{aligned}
	T_1 &\leq \frac{C}{\abs{k}^n} \myint{0}{t}{ \abs{e^{ikc(\phat-1)(\tau-t)} \left( 1 + \ordo((kh)^p) \right)  - 1 }}{\tau} \\
	&\leq \frac{C}{\abs{k}^n} \myint{0}{t}{ \left( \abs{e^{ikc(\phat-1)(\tau-t)} - 1} + \ordo((kh)^p) \right)}{\tau}.
\end{aligned}
\end{equation}
Noting that
\begin{equation}
	\abs{e^{i\beta} - 1} \leq \abs{\beta}
\end{equation}
for any real $\beta$, we obtain
\begin{equation} \label{eq:beta_simpl}
	\abs{e^{ikc (\phat-1)(\tau-t)}-1} \leq \abs{kc (\phat-1)(\tau-t)} = \abs{k} \ordo((kh)^p),
\end{equation}
where the estimate $\tau-t = \ordo(1)$ is valid because we consider the final time fixed. Using \eqref{eq:beta_simpl} in \eqref{eq:pre_beta} leads to
\begin{equation}
\begin{aligned}
	T_1 &\leq \frac{C}{\abs{k}^n} (\abs{k}\ordo((kh)^p) + \ordo((kh)^p)) =  \frac{1 + \abs{k}^{-1}}{\abs{k}^{n-1}} \ordo((kh)^p) .
\end{aligned}
\end{equation}
Because we are studying wavenumbers $\abs{k} \geq 2 \pi/L$, we have $\abs{k}^{-1} \leq L/2\pi = C$. It follows that
\begin{equation}
	T_1 =  \frac{1}{\abs{k}^{n-1}} \ordo((kh)^p).
\end{equation}
Since $g$ is $p+2$ times continuously differentiable we may set $n=p+2$, which yields 
\begin{equation}
	T_1 =  \frac{1}{\abs{k}^{p+1}} \ordo((kh)^p).
\end{equation}
Adding $T_1$ and $T_2$ leads to the error estimate
\begin{equation}
	\abs{\fcoeff{\bferr}_k} \leq T_1 + T_2 = \frac{1}{\abs{k}^{p+1}} \ordo((kh)^p) + \frac{1}{\abs{k}^{p+1}} \ordo((kh)^p) = \frac{1}{\abs{k}^{p+1}} \ordo((kh)^p).
\end{equation}
It follows that there exist constants $C$ and $\kappa_0$ such that, for $\abs{kh} \leq \kappa_0$,
\begin{equation}
	\abs{\fcoeff{\bferr}_k} \leq \frac{C}{\abs{k}^{p+1}} \abs{kh}^p = C\frac{h^p}{\abs{k}}.
\end{equation}
\\ \\
\noindent \textbf{Case 3:} $\abs{kh} > \kappa_0$ \\
In this case we prove that the error is small by proving that both $\fcoeff{\bfu}_k$ and $\fcoeff{u}_k$ are small. 
By \eqref{eq:fcoeff_decay}, we have
\begin{equation}
	\abs{\fcoeff{u}_{k}} \leq \frac{C}{\abs{k}^n} = C \frac{h^p}{\abs{kh}^p \abs{k}^{n-p}} \leq C \frac{h^p}{\abs{k}^{n-p}}.
\end{equation}
Setting $n=p+1$ yields the desired estimate
\begin{equation}
	\abs{\fcoeff{u}_{k}} \leq C \frac{h^p}{\abs{k}}.
\end{equation}

Now consider $\fcoeff{\bfu}_k$. For $\abs{k} \geq k_*$, we have $F(kh)=0$ and hence $\fcoeff{\bfu}_k=0$ according to the solution formula \eqref{eq:fcoeff_discr_2}. For $\abs{k} < k_*$, \eqref{eq:fcoeff_discr_ibp} states that
\begin{equation}
	\fcoeff{\bfu}_k  = F(kh) \frac{(-1)^n}{L(ik(c\phat-v_0))^n} \myint{0}{t}{ g^{(n)}(\tau) e^{ik((c\phat-v_0)\tau -c \phat t)} }{\tau}.
\end{equation}
Taking the absolute value yields
\begin{equation}
\begin{aligned}
	\abs{\fcoeff{\bfu}_k} &\leq \frac{\abs{F(kh) } }{L\abs{k}^n\abs{c\phat-v_0}^n} \myint{0}{t}{ \abs{g^{(n)}(\tau)}  }{\tau} \leq \frac{C \abs{F(kh) } }{\abs{k}^n \abs{c\phat-v_0}^n} \\
	&= \frac{C h^{n-1}}{\abs{kh}^{n-1} \abs{k}} \frac{\abs{F(kh)}}{\abs{c\phat-v_0}^n} \leq \frac{C h^{n-1}}{\abs{k}} \frac{\abs{F(kh)}}{\abs{c\phat-v_0}^n}
\end{aligned}
\end{equation}
Setting $n=p+1$ yields
\begin{equation}
	\abs{\fcoeff{\bfu}_k}  \leq C \frac{h^p}{\abs{k}} \frac{\abs{F(kh) } }{\abs{c\phat-v_0}^{p+1}}.
\end{equation}
Because $F$ satisfies $p+1$ sonic boom conditions, we may apply Lemma \ref{lemma:symbol_at_boom} to obtain the estimate
\begin{equation}
	\abs{\fcoeff{\bfu}_k}  \leq C \frac{h^p}{\abs{k}}.
\end{equation}
The error satisfies
\begin{equation}
	\fcoeff{\bferr}_k = \abs{\fcoeff{\bfu}_k - \fcoeff{u}_k} \leq \abs{\fcoeff{\bfu}_k} + \abs{\fcoeff{u}_k} \leq C \frac{h^p}{\abs{k}} + C \frac{h^p}{\abs{k}} \leq C \frac{h^p}{\abs{k}}.
\end{equation}

\end{proof}

\begin{remark}
Recall that the sonic boom conditions transition to the smoothness conditions of Petersson et al.\ if $v_0=0$.  Petersson et al.\ showed that it is possible to derive an almost equally strong bound on $\abs{\fcoeff{\bfu}_k}$ with only $p$ smoothness conditions, at the cost of slightly more involved analysis. In theory we could have taken a similar approach, but with the motion-consistent source discretization there is no obvious drawback of requiring more sonic boom conditions. Hence, for simplicity, we here require $p+1$ conditions.

\end{remark}

\begin{theorem} \label{thm:physical}
Under the same assumptions as in Theorem \ref{thm:fourier}, the error in the physical domain,
\begin{equation}
	\bferr_j = \bfu_j - u|_{x=x_j},
\end{equation}
satisfies
\begin{equation}
	\norm{\bferr}_h \leq Ch^{p},
\end{equation}
for some constant $C$.
\end{theorem}
\begin{proof}

The error at grid point $x_j$ is
\begin{equation}
\begin{aligned}
	\bferr_j &= \bfu_j - u|_{x=x_j} = \sum \limits_{k \in K_N}^{} \fcoeff{\bfu}_k e^{i k x_j} - \sum \limits_{k \in K_{\infty}} \fcoeff{u}_{k} e^{ik x_j} \\
	&= \sum \limits_{k \in K_N} (\fcoeff{\bfu}_k - \fcoeff{u}_k)e^{i k x_j} - \sum \limits_{k \in K_{\infty} \setminus K_{N} } \fcoeff{u}_{k} e^{ik x_j} \\
	&= \sum \limits_{k \in K_N} \fcoeff{\bferr}_ke^{i k x_j} - \sum \limits_{k \in K_{\infty} \setminus K_{N} } \fcoeff{u}_{k} e^{ik x_j} .
\end{aligned}
\end{equation}
By the Plancherel theorem,
\begin{equation} \label{eq:plancherel}
\ip{\bferr}{\bferr}{h} = \sum \limits_{k \in K_N} \abs{\fcoeff{\bferr}_k}^2 + \sum \limits_{k \in K_{\infty} \setminus K_{N} } \abs{\fcoeff{u}_k}^2.
\end{equation}
Using the bound $\abs{\fcoeff{u}_k} \leq C \abs{k}^{-(p+1)}$, which follows from \eqref{eq:fcoeff_decay} with $n=p+1$, the second sum in \eqref{eq:plancherel} can be bounded as
\begin{equation}
\begin{aligned}
	\sum \limits_{k \in K_{\infty} \setminus K_{N} } \abs{\fcoeff{u}_k}^2 &\leq C \sum \limits_{k \in K_{\infty} \setminus K_{N} } \frac{1}{\abs{k}^{2p+2}} = Ch^{2p} \sum \limits_{k \in K_{\infty} \setminus K_{N} } \frac{1}{\abs{k}^{2} \abs{kh}^{2p}} \\
	&\leq Ch^{2p} \sum \limits_{k \in K_{\infty} \setminus K_{N} } \frac{1}{\abs{k}^{2} } =  Ch^{2p} \sum \limits_{\abs{m} = N+1}^{\infty} \frac{1}{\abs{\frac{2\pi m}{L}}^{2} } \\
	&= C h^{2p} \sum \limits_{\abs{m} = N+1}^{\infty} \frac{1}{m^2 } \leq C h^{2p},
\end{aligned}
\end{equation}
where we used that the series converges in the last step. The first sum in \eqref{eq:plancherel} satisfies
\begin{equation}
\begin{aligned}
	\sum \limits_{k \in K_N} \abs{\fcoeff{\bferr}_k}^2 &\leq \sum \limits_{k \in K_N \setminus \{0 \} } \left( C \frac{h^p}{\abs{k}} \right)^2 = Ch^{2p} \sum \limits_{\abs{m} = 1}^{N} \frac{1}{\abs{\frac{2\pi m}{L}}^{2}} \\
	&= Ch^{2p} \sum \limits_{\abs{m} = 1}^{N} \frac{1}{m^2}
	\leq Ch^{2p}.
\end{aligned}
\end{equation}
We conclude that
\begin{equation} \label{eq:err_bound_squared}
	\norm{\bferr}^2_h = \ip{\bferr}{\bferr}{h} \leq Ch^{2p} + Ch^{2p} = Ch^{2p},
\end{equation}
and the final result follows after taking the square root of \eqref{eq:err_bound_squared}.
\end{proof}

\section{Windowing the source}\label{sec:window}
For a source discretization to be applicable in a finite domain, it must have compact support in physical space. While the motion-consistent source discretization formally has global support, the magnitude of the source coefficients $\bfd^{x_0}_j$ decays rapidly as the distance $\abs{x_j-x_0}$ increases, see Figure \ref{fig:compoldandnewsources}. We will show numerically that we can window the source to a finite width, which decreases with grid refinement, without reducing the order of accuracy. The procedure is similar to the truncation of the source in \cite{Zahedi2010}.

To analyze the decay rate of $\bfd^{x_0}_j$ we introduce the Fourier interpolant,
\begin{equation} \label{eq:def_d}
	d^{x_0}(x) = \sum \limits_{k \in K_N} \fcoeff{\bfd}^{x_0}_k e^{ikx} = \frac{1}{L} \sum \limits_{k \in K_{N}} F(kh) e^{ik(x-x_0)},
\end{equation}
which satisfies $d^{x_0}(x_j) = \bfd^{x_0}_j$.
Since $F(\kappa)$ is zero for $\abs{\kappa}\geq \pi$, we may extend the sum in \eqref{eq:def_d} to infinity:
\begin{equation}
d^{x_0}(x) =  \frac{1}{L} \sum \limits_{k \in K_{\infty}} F(kh) e^{ik(x-x_0)} .
\end{equation}
To provide some intuition for the decay rate of $d^{x_0}$, let us consider the (non-periodic) function
\begin{equation}
	D^{x_{0}}(x) = \frac{1}{2\pi L} \myint{-\infty}{\infty}{F(kh)e^{ik(x-x_0)}}{k}.
\end{equation}
The Fourier \emph{transform} of $D^{x_{0}}$ is
\begin{equation}
	\mathcal{F}[D^{x_{0}}](k) = \frac{e^{-ikx_0}}{L} F(kh).
\end{equation}
Notice that $\fcoeff{d}^{x_0}_k=\mathcal{F}[D^{x_{0}}](k)$, which leads us to expect $d^{x_0} \approx D^{x_0}$.
If $F$ satisfies $q$ moment conditions and $q$ sonic boom conditions (i.e., $m=s=q$), applying the integration-by-parts formula $q$ times yields
\begin{equation}
	D^{x_{0}}(x) = \frac{1}{2\pi L} \left(\frac{-h}{i(x-x_0)}\right)^q \myint{-\infty}{\infty}{F^{(q)}(kh)e^{ik(x-x_0)}}{k}.
\end{equation}
It follows that
\begin{equation}
\begin{aligned}
\abs{D^{x_{0}}(x)} &\leq \frac{1}{2\pi L} \frac{h^q}{\abs{x-x_0}^q} \myint{-\pi/h}{\pi/h}{\abs{F^{(q)}(kh)e^{ik(x-x_0)}}}{k} \\
&\leq \frac{1}{2\pi L} \frac{h^q}{\abs{x-x_0}^q} \frac{2\pi}{h} C = C \frac{h^{q-1}}{\abs{x-x_0}^q},
\end{aligned}
\end{equation}
where we used that $F^{(q)}(\kappa) = 0$ for $\abs{\kappa} \geq \pi$ and $F^{(q)}$ is assumed to be bounded. Assuming that the same decay rate holds for $d^{x_0}$ leads to the conjecture
\begin{equation}
\label{eq:assumedRelationd}
\abs{d^{x_{0}}(x)} \leq   C \frac{h^{q-1}}{\abs{x-x_0}^q},
\end{equation}
for some constant $C$. We will verify that \eqref{eq:assumedRelationd} holds numerically.

Let $W$ denote the rectangular window of width $2 \ell$:
\begin{equation}
\label{eq:window}
    W(\xi) = \left\{ \begin{array}{cc}
	1, & \abs{\xi} < \ell \\
	0, & \abs{\xi} \geq \ell
	\end{array} \right.
\end{equation}
Assuming that the conjecture \eqref{eq:assumedRelationd} holds, the error introduced by windowing $d^{x_0}$ is
\begin{equation}
	e(x) = \abs{W(x-x_0)d^{x_0}(x) - d^{x_0}(x)} \leq C \frac{h^{q-1}}{\ell^q}.
\end{equation}
Note that the error vector arising from the windowing satisfies
\begin{equation}
	\mathbf{e}_j = \abs{W(x_j-x_0) \bfd^{x_0}_j - \bfd^{x_0}_j} = \abs{W(x_j-x_0) d^{x_0}(x_j) - d^{x_0}(x_j)} = e(x_j).
\end{equation}
Setting $\ell = C_\ell h^w$, where $C_\ell$ is a constant, yields
\begin{equation}
\label{eq:relation:epqw}
	\mathbf{e}_j \leq C_{\ell} h^{q-1-qw}.
\end{equation}
To make $\mathbf{e}_j$ of order $p$, we need
\begin{equation}
	q-1-qw \geq p \Leftrightarrow w \leq \frac{q-1-p}{q}.
\end{equation}
Given $w$ and $p$ this implies the following requirement on $q$:
\begin{equation} \label{eq:qwp}
	q \geq \frac{p + 1}{1-w}.
\end{equation}
Clearly, the ideal choice $w=1$, which yields a source width proportional to $h$, is not possible. As an example, setting $w=1/2$ yields the condition
\begin{equation}
	q \geq 2p + 2.
\end{equation}
Note that a width proportional to $h^w$ implies a number of nonzero coefficients proportional to $h^{w-1}$.

\subsection{Numerical verification}
In this section, the conjectured bound on $\mathbf{e}_j$ in terms of $q$ and $w$, stated in \eqref{eq:relation:epqw}, is verified by numerical experiments. We construct motion-consistent discrete $\delta$ distributions for various parameters and compute the error vector $\mathbf{e}$ by comparing the windowed $\delta$ discretization, $W\bfd^{x_0}$, with the global version, $\bfd^{x_0}$.

We set $L = 1$, and $x_0 = \frac{1}{2}+ \frac{1}{29}$. In all experiments, the constant $C_\ell$ is chosen such that $\ell = 1/2$ when $h = 1/16$. Figure \ref{fig:convCut} shows the $l^2$ norm of $\mathbf{e}$ for $w = 1/4$ and $w = 3/4$, with $k_*h = \pi$.  In all cases, a slightly higher $p$ is observed than the conjecture \eqref{eq:relation:epqw} suggests, which indicates that the conjectured bound is valid, although perhaps not sharp. 

To further verify the conjectured bound, Tables \ref{tab:convcutFpi} and \ref{tab:convcut075Fpi} show the observed convergence rate $p$ in numerical experiments with various choices of $w$, for $ k_*h= \pi$ and $ k_*h= \frac{3}{4}\pi$. The observed $p$ is computed by taking the average of 
\begin{equation}
 p_{h} = \log_2 \left(\frac{e_{2h}}{e_{h}} \right),
 \end{equation}
for $h \in \left\{2^{-4}, \ 2^{-5}, \ 2^{-6}, \ \dots, \ 2^{-10}\right\}$, where $e_{h}$ denotes the $l^2$ norm of the error vector obtained with grid spacing $h$. Errors smaller than $10^{-9}$ were excluded to avoid values dominated by floating point errors. The observed rates further indicate that the conjectured bound is valid (but perhaps not sharp).

\begin{figure}[hbt]
   \subfigure[$w = 1/4$]{\label{subfig:cutw025}
    \includegraphics[width=0.475\textwidth]{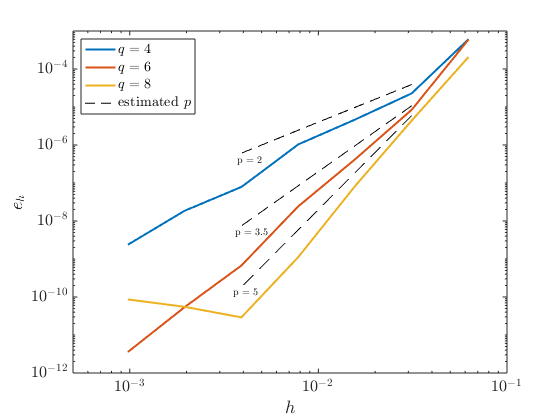}}
  \subfigure[$w = 3/4$]{\label{subfig:cutw075}
    \includegraphics[width=0.475\textwidth]{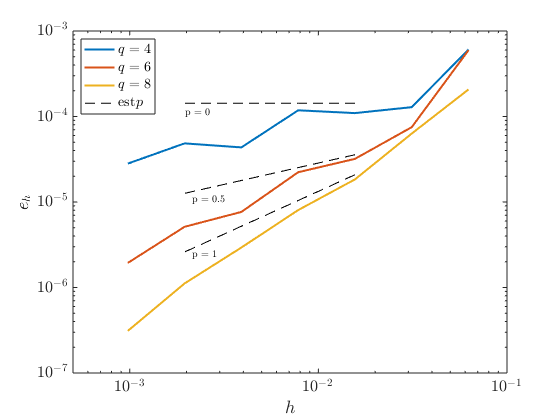}}
  \caption{Numerical verification of the conjectured relation between $p$, $q$ and $e$. Dashed lines indicate the conjectured rates, while solid lines show the experimentally observed errors.}
  \label{fig:convCut}
\end{figure}

\begin{table}[H]
    \center
	\begin{tabular}{c|lllll}
		\diagbox{$q$}{$w$} & $\frac{1}{4}$& $\frac{1}{3}$ & $\frac{1}{2}$ & $\frac{2}{3}$ & $\frac{3}{4}$ \\
		\hline
		4 & 2.9 (2.0) & 2.6 (1.7) & 1.9 (1.0) & 1.1 (0.3)  & 0.7 (0.0)\\
		6 & 4.9 (3.5) & 4.4 (3.0) & 3.0 (2.0) & 1.9 (1.0)  & 1.4 (0.5)\\
		8 & 5.8 (5.0) & 5.1 (4.3) & 3.9 (3.0) & 2.3 (1.7) & 1.6 (1.0)\\
	\end{tabular}
	\caption{Observed and conjectured (in parentheses) $p$ with $k_*h = \pi$}
	\label{tab:convcutFpi}
\end{table}

\begin{table}[H]
	\center
	\begin{tabular}{c|lllll}
		\diagbox{$q$}{$w$} & $\frac{1}{4}$& $\frac{1}{3}$ & $\frac{1}{2}$ & $\frac{2}{3}$ & $\frac{3}{4}$ \\
		\hline
		4 & 2.9 (2.0) & 2.6 (1.7) & 1.8 (1.0) & 1.0 (0.3) & 0.7 (0.0)\\
		6 & 4.6 (3.5) & 4.0 (3.0) & 2.9 (2.0) & 1.8 (1.0) & 1.2 (0.5)\\
		8 & 6.3 (5.0) & 5.5 (4.3) & 4.0 (3.0) & 2.5 (1.7) & 1.8 (1.0)\\
	\end{tabular}
	\caption{Observed and conjectured (in parentheses) $p$ with $k_*h = 0.75\pi$}
	\label{tab:convcut075Fpi}
\end{table}

\section{Implementation aspects}\label{sec:implementation}
There are some design choices to be made when implementing the motion-consistent source discretization. In this section, we summarize the requirements on the source discretization and present the implementation used in the numerical experiments.

Recall that the motion-consistent discrete $\delta$ distribution is defined by its Fourier coefficients:
\begin{equation}
\fcoeff{\bfd}^{x_0}_k = \frac{e^{-ikx_0}}{L} F(kh),
\end{equation}
where we have chosen to define $F$ as in \eqref{eq:F_poly}. To obtain convergence order $p$, $\bfd^{x_0}$ needs to satisfy $m=p$  moment conditions and $s=p+1$ sonic boom conditions according to Theorem \ref{thm:fourier}. For simplicity, we set $m=s=q$ and require $q \geq p+1$. Additionally, to be able to use a window of width $\ell = C_\ell h^w$ without reducing the convergence rate, we need to satisfy \eqref{eq:qwp}. In our experiments, we opt for $w=1/2$, which corresponds to a source discretization that covers on the order of $\sqrt{N}$ grid points. For $w=1/2$, the requirement \eqref{eq:qwp} becomes $q \geq 2p+2$. To satisfy both the requirements from Theorem \ref{thm:fourier} and the requirement from the windowing, we set
\begin{equation}
	q = \max(p+1, 2p+2) = 2p+2.
\end{equation}
 
 In the case of an accelerating source, we define the sonic boom wavenumber $k_*$ as the smallest positive solution to $c\hat{P}(kh) = v_{max}$, where $v_{max}$ denotes the highest source velocity during the simulation. If $ c\hat{P}(kh) = v_{max}$ does not have a solution in the interval $0 \leq kh \leq \pi$, no sonic boom conditions are required. Using a higher velocity than occurs in the simulation (e.g., $\gamma v_{max}$, $\gamma >1$) to define $k_*$ does not reduce the convergence rate of the numerical method. However, using an excessively high velocity will exclude more wavenumbers than necessary and might therefore affect the accuracy for a given $h$. 

The Fourier interpolant $d^{x_0}$ of $\bfd^{x_0}$ is shown for different combinations of $q$ and $k_*$ in Figure \ref{fig:showmovablesources}. Figure \ref{subfig:diffq} shows $d^{x_0}$ for different $q$, with $k_*h = \pi$. Figure \ref{subfig:diffk} shows $d^{x_0}$ for different $k_*$, with $q = 14$.
\begin{figure}[hbt]
	\subfigure[$k_*h = \pi$]{\label{subfig:diffq}
		\includegraphics[width=0.475\textwidth]{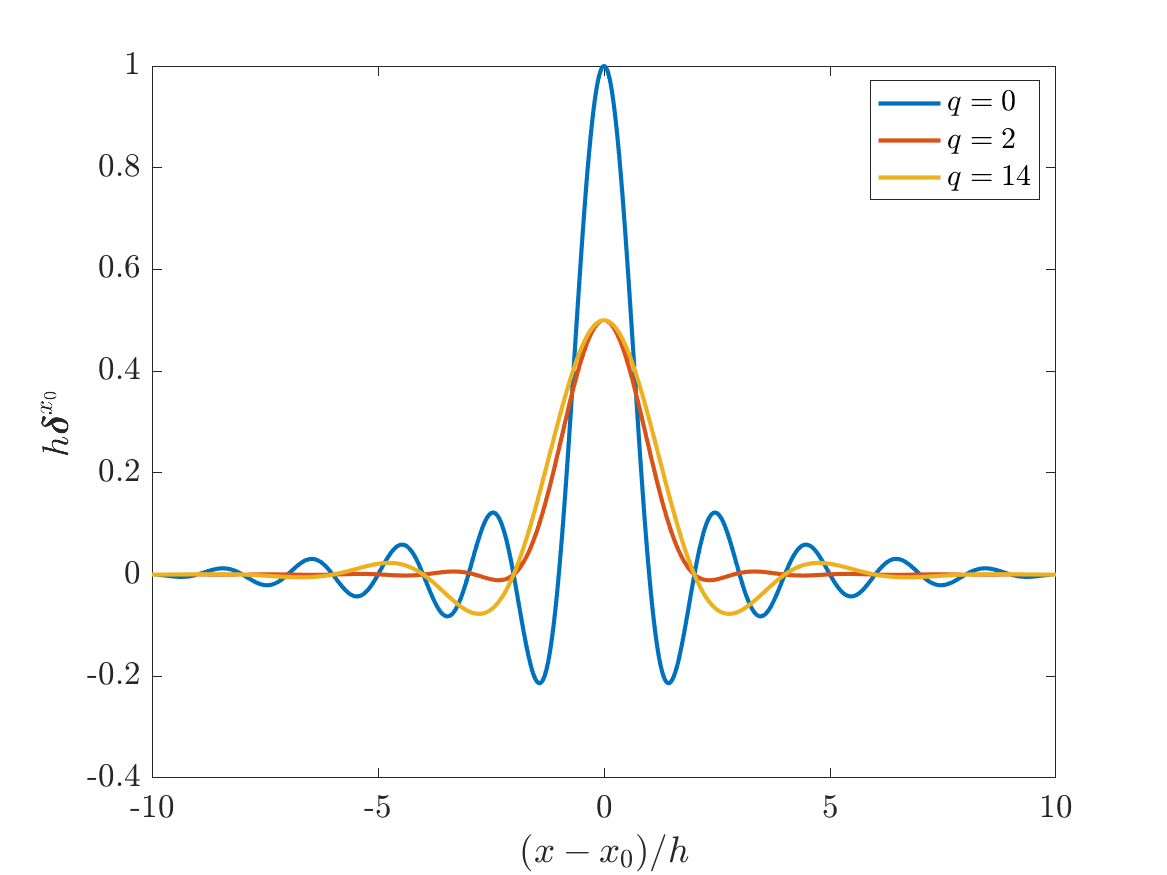}}
	\subfigure[$q = 14$]{\label{subfig:diffk}
		\includegraphics[width=0.475\textwidth]{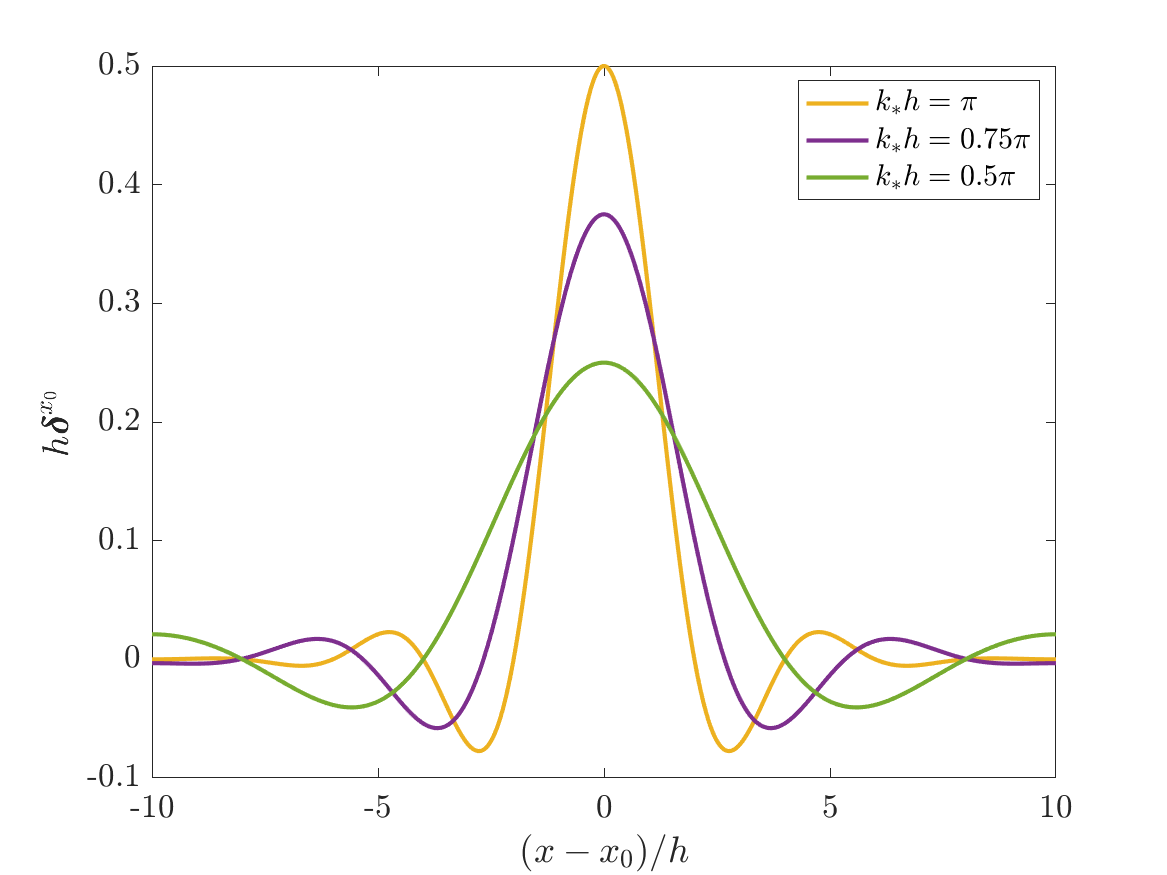}}
	\caption{The motion-consistent discrete $\delta$ distribution for different values of $k_*$ and $q$.}
	\label{fig:showmovablesources}
\end{figure}

Since $\bfd^{x_0}$ is defined in terms of its Fourier coefficients, evaluating it on the grid by naively summing over Fourier modes requires on the order of $N^2$ operations. In our implementation we utilize the fast Fourier transform, which reduces the complexity to $N\log(N)$. When the source is windowed, the source discretization is computed for the whole domain and values outside the window are discarded. In applications where the source is far from boundaries, the user might decide to use a larger window than we have suggested here, to minimize truncation errors from the windowing.

\section{Numerical experiments in 1D} \label{sec:numexp1D}
In this section we verify the convergence rate of the windowed source discretization when solving the advection equation \eqref{eq:ut_ux} with constant source velocity. We choose the source trajectory $x_0(t) = 1 + v_0t$. This problem has the exact solution
\begin{equation}
	\begin{aligned}
		u = \begin{cases}
			\frac{g(\tau)}{c-v_0}, &\tau > 0,\\
			0, & \tau \leq 0.
		\end{cases}
	\end{aligned}
\end{equation}
where
\begin{equation}
\tau = \frac{ct- (x-1)}{c-v_0}.
\end{equation}
We select the parameters $c =1$ and $v_0 = 0.5$. To ensure a smooth solution we use a Gaussian source time function 
 \begin{equation} 
 \label{eq:Gauss1}
 	g(t) = \frac{1}{\sigma \sqrt{2 \pi}}e^{-\frac{(t-t_0)^2}{2\sigma^2}},
 \end{equation} 
 with $t_0 = 1$ and $\sigma = 0.15$. The error is computed at time $t_{end} = 2$. The exact solution at $t_{end}$ and the source path are displayed in Figure \ref{fig:sourceComparison1D}. The test problem is discretized in space on the periodic domain $x \in [0,\ 4]$ with central finite differences of order $p = 2,4,6$ according to \eqref{eq:semidisc_ut_ux}. We use the classical fourth order Runge--Kutta method with time step $\Delta t = 0.2h$ for time integration.
 
 In the experiment, the point source discretion has been truncated by the window in \eqref{eq:window} with $w=1/2$, $q = 2p+2$ and $C_\ell = 4$. Our theoretical results indicate that this combination of $q$ and $w$ yields a discretization of convergence order $p$. The errors and convergence rates are displayed in Table \ref{tab:conv1D}.
 The convergence rate $c_{r}$ is calculated as
 \begin{equation}
 c_{r} = \log \left( \frac{e_{M_2}}{e_{M_1}} \right)\big / \log \left( \frac{M_2}{M_1} \right), 
 \end{equation}
 where $e_M$ is the $l^2$ norm of the error vector obtained with $M$ grid points. 
Table \ref{tab:conv1D} shows that the experimental convergence rates agree well with the theoretical results. 

\begin{table}[htb]
	\center
	\begin{tabular}{r  |c c| c c| c c}
		& \multicolumn{2}{c|}{$2nd$ order} & \multicolumn{2}{c|}{$4th$ order} & \multicolumn{2}{c}{$6th$ order}\\
		\hline
		$M$ & $log_{10}(e_M)$ &$c_{r}$&  $log_{10}(e_M)$ &$c_{r}$&$log_{10}(e_M)$ &$c_{r}$\\
		100  &-0.30& -    & -0.94 & -     & -1.37 & -    \\
		200  &-0.65& 1.19 & -1.93 & 3.32  & -3.01 & 5.50 \\
		400  &-1.25& 1.99 & -3.11 & 3.93  & -4.75 & 5.81 \\
		800  &-1.86& 2.04 & -4.31 & 3.99  & -6.55 & 5.97 \\
		1600 &-2.47& 2.01 & -5.51 & 4.00  & -8.35 & 5.98 \\
	\end{tabular}
	\caption{$l^2$ errors and convergence rates for the 1D advection equation with $v_0 = 0.5$ }
	\label{tab:conv1D}
\end{table}

\begin{figure}[htb]
	\centering
		\includegraphics[width=0.75\textwidth]{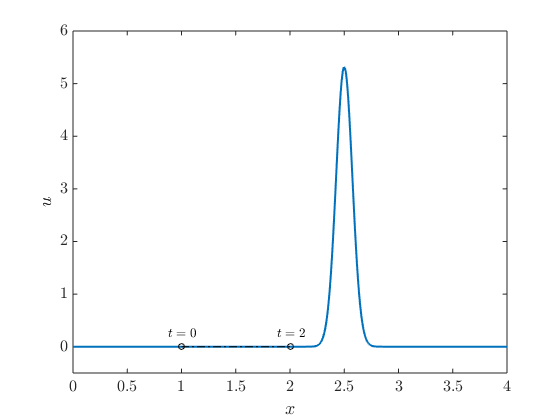}
	\caption[ ]
	{Exact solution and source path for the wave equation in 1D}
	\label{fig:sourceComparison1D}
\end{figure}

\section{Numerical experiments in 2D} \label{sec:numexp2D}
In this section we solve the acoustic wave equation in two dimensions,
\begin{equation}
\label{eq:waveeq2d}
\begin{aligned}
\rho \dd{\bar{v}}{t} + \nabla  \theta &= 0, \\
\frac{1}{K}\dd{\theta}{t} + \nabla \cdot \bar{v} &= g(t) \delta(\bar{x}-\bar{x}_0(t)),
\end{aligned}
\end{equation}
where $\bar{x} = [x, \ y]^\top$ is position, $\theta$ is pressure, and $\bar{v}$ is particle velocity. We consider a square domain with side $L=2.5$ and impose characteristic boundary conditions on all boundaries. We select the material parameters $K = 1$ and $\rho = 1$, which gives the wave speed $c = 1$. The source moves with constant speed $v_0=0.5$ along the circle given by
\begin{equation}
\bar{x}_0(t) =
\begin{bmatrix}
x_0(t) \\ y_0(t)
\end{bmatrix}
=
\begin{bmatrix} 
1.25 + 0.2\sin(5v_0t) \\ 
1.25 + 0.2\cos(5v_0t) 
\end{bmatrix},
\end{equation}
We use the Gaussian source time function \eqref{eq:Gauss1} with $t_0=1$ and $\sigma = \frac{1}{25}$. The problem is discretized in space with Summation-by-Parts finite-difference operators \cite{Svard2014,DelReyFernandez2014a} and the boundary conditions are imposed weakly, using Simultaneous Approximation Terms \cite{CarpenterGottlieb94}. The procedure is identical to the centered finite-difference discretization in \cite{Rydin2018}.

 We define the two-dimensional discrete $\delta$ distribution as the Cartesian product of two one-dimensional distributions corresponding to the coordinate directions,  
\begin{equation}
\bfd^{\bar{x}_0} = \bfd^{x_{0}} \kron \bfd^{y_{0}}.
\end{equation}
This tensor-product construction of multidimensional sources was presented for stationary sources by Tornberg and Engquist in \cite{Tornberg2004}. In both spatial directions, the sonic boom wavenumber $k_*$ has been defined by $v_{max} = v_0$. As in the previous experiment, $\bfd^{x_0}$ and $\bfd^{y_0}$ have been truncated by the window in \eqref{eq:window} with $w=1/2$, $q = 2p+2$ and $C_\ell = 4$.   

For time discretization, we again employ the fourth order Runge--Kutta method, with time step $\Delta t = 0.1h$. We study the error at time $t_{end} = 1$.  Rather than an analytical solution, we compare against a reference solution, shown in Figure \ref{fig:refsol2D}, computed with $N=3201$ grid points in each coordinate direction. The converge rate is computed as
 \begin{equation}
 c_{r} = \log \left( \frac{e_{N_2}}{e_{N_1}} \right)\big / \log \left( \frac{N_2-1}{N_1-1} \right)
 \end{equation}
where $e_N$ is the $l^2$ norm of the error vector obtained with $N$ points in each spatial direction. The obtained errors and convergence rates are displayed in Table $\ref{tab:conv2D}$. The convergence rates agree well with the theoretical and experimental results for one dimension. This indicates that the tensor product of motion-consistent discrete $\delta$ distributions yields design-order convergence for accelerating sources in multiple dimensions.

The accuracy of the finite difference scheme is reduced to order $p/2$ at boundaries. Therefore, to highlight the accuracy of the source discretization, we choose to end the simulation before the wavefield interacts with the boundaries. If the wavefield were allowed to interact with the boundaries, we would observe convergence of order $p/2+1$ \cite{Gustafsson75}. We stress, however, that there are no problems associated with the wavefield interacting with the boundaries that would not be present in a simulation without a point source. The only requirement is that $\bar{x}_0$ stays far enough from boundaries that the footprint of the windowed $\bfd^{\bar{x}_0}$ does not overlap with the boundary stencils of the Summation-by-Parts finite difference discretization. Since $\bfd^{\bar{x}_0}$ was designed for periodic problems, it is only allowed to act on grid points where the centered finite difference operator is applied. As long as the source trajectory does not touch the boundary, this condition is always satisfied for $h$ smaller than some $h_0$, since the source width decreases with grid refinement. In practice, if the source path is very close to the boundary, grid-refining until $h < h_0$ may of course be infeasible, however.

\begin{table}[H]
	\center
	\begin{tabular}{r  |c c| c c| c c}
		 & \multicolumn{2}{c|}{$2nd$ order} & \multicolumn{2}{c|}{$4th$ order} & \multicolumn{2}{c}{$6th$ order}\\
		\hline
		$N$ & $log_{10}(e_N)$ &$c_{r}$&  $log_{10}(e_N)$ &$c_{r}$&$log_{10}(e_N)$ &$c_{r}$\\
		101  &-0.70& -& -0.87 & -    & -0.95 & - \\
		201  &-0.75& 0.17& -1.24 & 1.26 & -1.58 & 2.11 \\
		401  &-0.94& 0.62& -2.00 & 2.82 & -3.06 & 4.94 \\
		801  &-1.41& 1.57& -3.25 & 3.88 & -4.87 & 6.04 \\
		1601 &-2.01& 1.98& -4.45 & 3.99 & -6.68 & 5.99 \\
	\end{tabular}
	\caption{$l^2$ errors and convergence rates for the 2D wave equation }
	\label{tab:conv2D}
\end{table}

\begin{figure}[htb]
		\includegraphics[width=\textwidth]{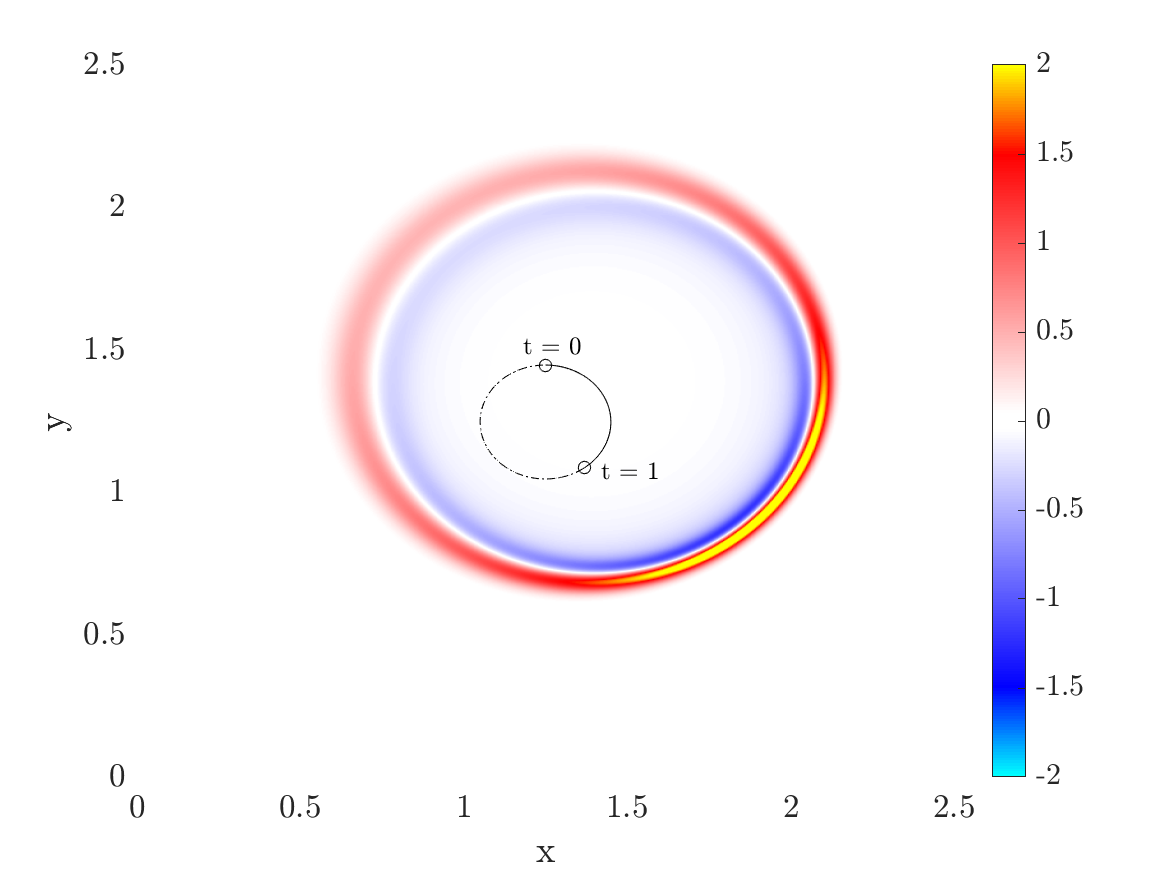}
	\caption[ ]
	{Pressure component of the reference solution and source trajectory for the 2D wave equation}
	\label{fig:refsol2D}
\end{figure}

\section{Conclusion} \label{sec:conclusion}
We have derived high-order discretizations of moving point sources in hyperbolic equations. Compared to a stationary source discretization, there are additional requirements on the moving source for convergence. First, the source must not excite modes that propagate with the same velocity as the source. Second, the Fourier spectrum amplitude of the discrete $\delta$ distribution needs to be independent of the source posititon. The convergence properties of the source discretization are verified by numerical experiments with the advection equation in one dimension and with an accelerating source in the acoustic wave equation in two dimensions. The approximation of the moving source covers on the order of $\sqrt{N}$ grid points on an $N$-point grid and is therefore not applicable if the source trajectory is very close to boundaries or interfaces; we hope to address this in future work.


\bibliographystyle{habbrv}
\bibliography{refs}

\end{document}